\documentclass{amsart}
\usepackage{amssymb}
\usepackage[all]{xy}
\usepackage{graphicx, hhline}
\usepackage[usenames]{color}
\usepackage{hyperref}
\hypersetup{colorlinks=true}

\newcommand{\Q}{\mathbb{Q}}
\newcommand{\R}{\mathbb{R}}
\newcommand{\C}{\mathbb{C}}
\newcommand{\Z}{\mathbb{Z}}
\newcommand{\N}{\mathbb{N}}
\newcommand{\m}{\mathfrak{m}}
\newcommand{\fkm}{\mathfrak{m}}

\newcommand{\fka}{\mathfrak{a}}
\newcommand{\fkp}{\mathfrak{p}}
\newcommand{\F}{\mathbb{F}}
\newcommand{\ba}{\mathfrak{a}}
\newcommand{\bb}{\mathfrak{b}}
\newcommand{\n}{\mathfrak{n}}
\newcommand{\sO}{\mathcal{O}}
\newcommand{\Div}{\mathrm{div}}

\newcommand{\Hom}{\mathrm{Hom}}
\newcommand{\Spec}{\mathop{\mathrm{Spec}}\nolimits}
\newcommand{\Proj}{\mathop{\mathrm{Proj}}\nolimits}

\newtheorem{thm}{Theorem}[section]
\newtheorem{lem}[thm]{Lemma}
\newtheorem{cor}[thm]{Corollary}
\newtheorem{prop}[thm]{Proposition}
\newtheorem{fact}[thm]{Fact}

\theoremstyle{definition}
\newtheorem{defn}[thm]{Definition}
\newtheorem{rem}[thm]{Remark}
\newtheorem{ex}[thm]{Example}
\newtheorem{conj}[thm]{Conjecture}
\newtheorem*{assumption}{Assumption}
\newtheorem*{conjA}{Conjecture $\mathrm{A}_n$}
\newtheorem*{conjB}{Conjecture $\mathrm{B}_n$}

\begin{document}

\title{$F$-singularities: applications of characteristic \\$p$ methods to singularity theory}
\author{Shunsuke Takagi}
\address{Graduate School of Mathematical Sciences, University of Tokyo, Komaba 3-8-1, Meguro-ku, Tokyo 153-8914, Japan}
\thanks{This article originally appeared in Japanese in Sugaku \textbf{66} (1) (2014), 1--30.}
\subjclass[2010]{13A35, 14B05}

\email{stakagi@ms.u-tokyo.ac.jp}
\author{Kei-ichi Watanabe}
\address{Department of Mathematics, College of Humanities and Sciences,
Nihon University, Setagaya-Ku, Tokyo 156-0045, Japan} 
\email{watanabe@math.chs.nihon-u.ac.jp}

\maketitle
\markboth{S.~TAKAGI and K.-i.~WATANABE}{$F$-SINGULARITIES}

\tableofcontents

\section{Introduction}

Let $R$ be a commutative ring containing a field of characteristic $p>0$. 
The Frobenius map $F:R\to R$ is a ring homomorphism sending $x$ to $x^p$ for each $x \in R$. 
The letter ``$F$" in the title of this article stands for the Frobenius map, and we mean by ``$F$-singularities" singularities defined in terms of the Frobenius map. 
The theory of $F$-singularities provides a series of methods to analyze singularities (not only in positive characteristic but also in characteristic zero), using techniques from commutative algebra in positive characteristic. 

The importance of analytic methods for singularity theory has been recognized for many years. 
 When the second-named author was a graduate student,  his supervisor Yukiyoshi Kawada gave  him the following advice: ``(some of) the deepest theorems in algebraic geometry have been proved by using analysis. It is better to study analytic methods." 
Indeed, the Kodaira vanishing theorem and the Brian\c con-Skoda theorem, which we will explain in \S\ref{tight closure}, were first proved by using analysis, and no purely algebraic proofs were known for these theorems until the 1980s.
Also, it was impressive to Watanabe that Hochster, who was a leading researcher in commutative ring theory,  gave a series of ten lectures entitled ``Analytic Methods in Commutative Algebra" at the NSF regional conference at George Mason University in 1979. 
It was especially impressive to him that Hochster lectured even on some theorems of functional analysis, carrying the famous thick book written by Griffiths and Harris \cite{GH}.  
Here is a quote from the introduction of the paper by Lipman and Teissier \cite{LT}: ``The proof given by Brian\c con and Skoda of this completely algebraic statement is based on a quite transcendental deep result of Skoda.  The absence of an algebraic proof has been for algebraists something of a scandal---perhaps even an insult---and certainly a challenge." 
This captures the feelings of commutative algebraists at the time very well. 

The situation changed dramatically in the mid 1980s. 
Deligne-Illusie \cite{DI} and Hochster-Huneke \cite{HH1} gave a comparatively elementary proof of the Kodaira vanishing theorem and the Brian\c con-Skoda theorem, respectively, using characteristic $p$ methods. 
Such a development of characteristic $p$ methods provided a framework to translate various notions about singularities in characteristic zero into the language of $F$-singularities. 
This leads us to the maxim: ``(As far as singularity theory is concerned) what has been proved by analytic methods can be proved by characteristic $p$ methods."  
The goal of this article is to exhibit the effectiveness of characteristic $p$ methods in singularity theory. 

We shall have a look at the organization of this article. 
In Section \ref{tight closure}, we review the theory of tight closure introduced by Hochster and Huneke. We then discuss the Boutot-type theorem for $F$-singularities and the tight closure version of the Brian\c con-Skoda theorem. 
In Section \ref{classical Fsing}, we overview four classes of $F$-singularities: strongly $F$-regular, $F$-rational, $F$-pure and $F$-injective rings. In particular, we observe that rational singularities and $F$-rational rings are ``morally equivalent." 
In Section \ref{Fpairs}, we generalize the definition of $F$-singularities to the pair setting. We then explain a correspondence of $F$-singularities and singularities in the minimal model program. We also mention the theory of $F$-adjunction introduced by Karl Schwede. 
In Section \ref{test ideals}, we explain two applications of asymptotic test ideals, a positive characteristic analogue of asymptotic multiplier ideals. One is to symbolic powers of ideals and the other is to asymptotic base loci in positive characteristic. 
In Section \ref{HK section}, we give an overview of Hilbert-{K}unz theory.  
We explain a characterization of regular local rings in terms of Hilbert-Kunz multiplicity and then discuss a lower bound for the Hilbert-Kunz multiplicity of any non-regular local ring. 
Finally, we close this article by listing some of the topics that are not discussed in this article. 


Throughout this article, all rings are Noetherian commutative rings with unity. 
For a ring $R$, we denote by $R^{\circ}$ the set of elements of $R$ which are not in any minimal prime ideal. 
 

\section{Tight closure}\label{tight closure}

The following facts are fundamental in the theory of singularities of algebraic varieties over a field of characteristic zero. 
\begin{enumerate}
\item (\cite{KKMS}) Rational singularities are Cohen-Macaulay. 
\item (Boutot's theorem \cite{Bo}) Pure subrings of rational singularities are again rational singularities. 
\item (Brian\c con-Skoda theorem \cite{BrS}) If $I$ is an ideal in a regular ring generated by $n$ elements, then the integral closure of $I^n$ is contained in $I$. 
\end{enumerate}
We can formulate analogous results in positive characteristic, using the theory of tight closure. 
Tight closure is a closure operation defined on ideals (and modules) in positive characteristic, and its name comes from the fact that tight closure is ``tighter" (smaller) than integral closure. 
In this section, we review the theory of tight closure and then discuss a tight closure version of the above theorems. 
The results, not specifically mentioned,  come from \cite{HH0}.  
We recommend  \cite{Hu2}, \cite{Hu3} for a nice introduction to the theory of tight closure. 
The reader is referred to \cite{BH}, \cite{GW1}, \cite{Ma} for standard notions and facts from commutative ring theory. 

Suppose that $p$ is a prime number and $R$ is a reduced ring of characteristic $p$, that is, $R$ contains the prime field $\F_p$. 
For a power $q=p^e$ of $p$ and an ideal $I$ of $R$, we put 
\[
I^{[q]}=( a^q | a\in I ) \subseteq R.
\]

\begin{defn}\label{tight closure def}
For an ideal $I$ of $R$, we define the \textit{tight closure} $I^*$ of $I$ as follows: 
an element $x \in R$ belongs to $I^*$ if and only if for there exists some $c\in R^{\circ}$ such that $cx^q \in I^{[q]}$ 
for all sufficiently large powers $q=p^e$ of $p$. 
We can easily show that $I^*$ is an ideal of $R$ containing $I$. 
We say that the ideal $I$ is \textit{tightly closed} if $I^*=I$. 
\end{defn}

In the theory of tight closure and $F$-singularities, Kunz's theorem plays a very important role. 
Before stating the theorem, we recall that the Frobenius map $F: R\to R$ is defined by $F(a)=a^p$ for every $a\in R$. 
Since $(x+y)^p=x^p+y^p$ in characteristic $p$, the map $F$ is a ring homomorphism. 

\begin{thm}[Kunz's Theorem \cite{Ku}]\label{Kunz} 
For a local ring $(R,\fkm)$ of dimension $d$ containing a field of characteristic $p>0$, the following conditions are equivalent.
\begin{enumerate}
\item $R$ is regular. 
\item The Frobenius map $F : R\to R$ is flat. 
\item $\ell_R( R / \fkm^{[p]})=p^d$. 
\item For any power $q= p^e$ of $p$, $\ell_R( R / \fkm^{[q]}) = q^d$, 
\end{enumerate}
where $\ell_R(M)$ denotes the length of an $R$-module $M$.
\end{thm}

\begin{prop}  If $(R, \m)$ is a regular local ring, then every ideal in $R$ is tightly closed.  
\end{prop}
\begin{proof} 
Suppose to the contrary that there exists an element $x \in I^* \setminus I$ for some ideal $I \subset R$. 
Then $(I : x) \subseteq \fkm$. 
Since the Frobenius map $F:R \to R$ is flat by Theorem \ref{Kunz}, 
$(I^{[q]} : x^q)=(I : x)^{[q]} \subseteq \fkm ^{[q]}$ for every $q=p^e$.\footnote{For ideals $I,J$ of a ring $A$, we denote $(I:J)=\{x\in A\;| xJ \subset I\}$.  
Also, we denote the ideal $(I: (y))$ simply by $(I :y)$.  This ``colon" operation will appear frequently in this section.}
On the other hand, by the definition of tight closure, there exists a nonzero element $c \in R$ such that $c x^q \in I^{[q]}$ for all sufficiently large $q=p^e$. 
Hence, we have 
\[
c\in \bigcap_{e \gg 0} (I^{[p^e]} : x^{p^e} ) \subseteq   \bigcap_{e \gg 0} \fkm ^{[p^e]}=(0), 
\]
which is a contradiction. 
Thus, $J^* = J$ for every ideal $J$ in $R$. 
\end{proof}

This result leads us to introduce the notion of $F$-regular rings.  
The notion of $F$-rational rings is also defined in a similar way. 
Although the definitions of $F$-rational rings and rational singularities are completely different at first glance, we can think of $F$-rational rings as a positive characteristic analogue of rational singularities as we will see in \S\ref{classical Fsing}. 

\begin{defn}\label{Frational} 
Let $R$ be a reduced ring of characteristic $p>0$. 
\begin{enumerate}
\item 
We say that $R$ is \textit{weakly $F$-regular} if every ideal in $R$ is tightly closed.
We say that $R$ is \textit{$F$-regular}  if every localization of $R$ is weakly $F$-regular.
\item (\cite{FW}) 
A local ring  $(R, \fkm)$ is said to be \textit{$F$-rational} if every parameter ideal\footnote{We call an ideal $I$ a \textit{parameter ideal} if $I$ is generated by $\mathrm{ht} \; I$ elements.} is tightly closed.
When $R$ is not local, we say that $R$ is $F$-rational if the local ring $R_{\fkp}$ is $F$-rational for every maximal ideal $\fkp$ of $R$. 
\end{enumerate}
\end{defn}

By definition, we have the following hierarchy of properties of local rings:
\[
\xymatrix{
\textup{$F$-regular} \ar@{=>}[r] & \textup{weakly $F$-regular} \ar@{=>}[r] & \textup{$F$-rational.}}
\]
If the ring is Gorenstein, then $F$-rationality implies $F$-regularity (see Proposition \ref{F-pure=>F-injective}).

\begin{rem}\label{localization remark}
 The question of whether tight closure commutes with localization, that is, whether $(I S^{-1}R)^* = I^* (S^{-1}R)$ for every multiplicative subset  $S\subset R$ and for every ideal $I$ of $R$, had been a long-standing open problem since tight closure was introduced by Hochster and Huneke in 1986. 
In 2006, Brenner and Monsky \cite{BM} gave a negative answer to this question by constructing a counterexample in characteristic 2. 
However, the question of whether every weakly $F$-regular ring is $F$-regular still remains open  (see also Remark \ref{F-regular remark} (1)). 
\end{rem}

\begin{ex}\label{Frational example}
Let $k$ be a field of characteristic $p>0$. 

(1) Let $R=k[X,Y,Z]/ (X^n + Y^n + Z^n)$, where $n \ge 2$ is an integer not divisible by $p$. 
Let $x, y, z$ denote the images of $X, Y, Z$ in $R$, respectively, and put $I = (y, z)$. 
Then $x \not\in I^*$ and $x^2 \in I^*$. 
Therefore, if $R$ is $F$-rational, then $n=2$, and the converse also holds if $p \ge 3$. 

(2) Suppose that  $R=k[X,Y,Z]/ (X^a + Y^b + Z^c)$ is reduced, where $a, b, c$ are integers greater than or equal to $2$. 
Let $x, y, z$ denote the images of $X, Y, Z$ in $R$, respectively, and put $I = (y,z)$.
Then the condition that $x^{a-1}\not\in I^*$ implies $1/a +1/b+1/c>1$. 
Therefore, if $R$ is $F$-rational, then $1/a +1/b+1/c>1$, and the converse also holds if $p>5$. 
This should be compared with the fact that when $k$ is a field of characteristic zero, $R$ is a rational singularity if and only if $1/a +1/b+1/c>1$. 
\end{ex}

We give a sketch of the proof of Example \ref{Frational example} in order to help the reader get a feeling of tight closure. 

\begin{proof}[A sketch of the proof of Example $\ref{Frational example}$]  
(1) Fix any $q=p^e$ and write  $q= nu +r$ with $0\le r<n$. 
Since $(x^2)^q =x^{2r}(x^n)^{2u}= (-1)^{2u}x^{2r}(y^n + z^n)^{2u}$, 
we have $(yz)^n(x^2)^q \in (y^q, z^q)=I^{[q]}$. Thus, $x^2 \in I^*$. 

In order to show $x \not\in I^*$, we use the notion of test elements (see Definition \ref{test element}). 
It follows from Lemma \ref{test vs jacobi} that there exists some power $z^m$ of $z$ such that for every $w \in I^*$, one has $z^m w^q \in I^{[q]}$ for all $q=p^e$.  

Suppose to the contrary that $x \in I^*$, and choose $q=p^e>m$. 
Let $S=k[X, Y, Z]$ be a polynomial ring, and put $J=(X^n+Y^n+Z^n, Y^q, Z^q) \subset S$. 
We take the graded reverse lexicographic order on $S$ with $X>Y>Z$.\footnote{The \textit{graded reverse lexicographic order} on the polynomial ring $k[X_1, \dots, X_n]$ is defined by saying that ${X}^{\alpha}>{X}^{\beta}$ if $\deg {X}^{\alpha}>\deg {X}^{\beta}$ or if $\deg {X}^{\alpha}=\deg {X}^{\beta}$ and in the vector difference $\alpha-\beta$, the rightmost non-zero entry is negative.} 
Then the initial ideal $\mathrm{in}(J)=(X^n, Y^q, Z^q)$. 
Write $q=nu+r$ with $0\le r<n$.  Since
\[\mathrm{in}((-1)^u Z^m X^r(Y^n+Z^n)^u)=(-1)^u X^r Y^{nu}Z^m \notin (X^n, Y^q, Z^q)=\mathrm{in}(J),\] 
one has $Z^m X^q \notin J$. This means that $z^m x^q \notin I^{[q]}$, which is a contradiction.

(2) Fix any $q=p^e$ and write  $q= au +r$ with $0\le r<a$. 
Suppose that $1/a +1/b+1/c \le 1$. 
Then
\[
(x^{a-1})^q =(-1)^{(a-1)u} x^{(a-1)r}(y^b + z^c)^{(a-1)u} \subset ((y^b+z^c)^{\lceil a/b+a/c \rceil u}),  
\] 
which implies that $(yz)^a(x^{a-1})^q \in (y^q, z^q)=I^{[q]}$. 
Thus, $x^{a-1} \in I^*$. 
\end{proof}

One of the important properties of $F$-rational rings is that they are Cohen-Macaulay.  
First we recall the definitions of Cohen-Macaulay rings and related notions. 

Let $(A,\fkm)$ be a local ring of dimension $d$ and $E_A(A/\m)$ be the injective hull of the residue field $A/\fkm$. 
Then the \textit{canonical module} $\omega_A$ of $A$ is defined by 
\[
\omega_A \otimes_A \widehat{A} \cong \Hom_A(H^d_{\fkm} (A),E_A(A/\fkm)).
\]
Alternatively, it is defined as $\omega_A=\mathcal{H}^{-d}(\omega^{\bullet}_A)$, where $\omega_A^\bullet$ is the normalized dualizing complex of $A$. 
In other words, 
\[
\omega_A^\bullet: 0 \to I^{-d} \to I^{-d+1} \to \cdots \to I^0 \to 0
\]
is a complex of $A$-modules satisfying the following two conditions: (1) $I^{-i} \cong \bigoplus_{\dim A/\fkp=i} E_A(A/\fkp)$ where $\fkp$ varies over all prime ideals in $A$ such that $\dim A/\fkp=i$ and $E_A(A/\fkp)$ is the injective hull of $A/\fkp$, (2) $\mathcal{H}^{-i}(\omega_A^\bullet)$ is a finitely generated $A$-module for each $i$. 
By this definition, if $A$ has a dualizing complex, then $A$ has a canonical module.\footnote{It is known in \cite{Kaw} that $A$ has a dualizing complex if and only if $A$ is a homomorphic image of a Gorenstein local ring.}

\begin{defn} Let $A$ be a $d$-dimensional local ring with canonical module $\omega_A$. 
\begin{enumerate}
\item We say that $A$ is a \textit{Cohen-Macaulay ring} if a full system of parameters 
$x_1, \dots, x_d$ for $A$ is a regular sequence, that is, if $(x_1,\ldots ,x_i) : x_{i+1} = (x_1,\ldots ,x_i)$ for $i=0, \dots, d-1$. 
Note that if some full system of parameters is a regular sequence, then so is every full system of parameters. 

\item We say that $A$ is a \textit{Gorenstein ring} If $A$ is Cohen-Macaulay and $\omega_A$ is a free $A$-module. 
We say that $A$ is a \textit{quasi-Gorenstein ring}  if $\omega_A$ is free.
A quasi-Gorenstein ring may not be Cohen-Macaulay. 
\end{enumerate}
When $A$ is not a local ring, we say that $A$ is Cohen-Macaulay (respectively, Gorenstein, 
quasi-Gorenstein) if so is the local ring $A_{\fkp}$ for every maximal ideal $\fkp$ of $A$.  
\end{defn}
Next we recall the definition of rational singularities. 

\begin{defn} Let $R$ be a normal ring essentially of finite type over a field $k$ of characteristic zero, that is, $R$ is a localization of a finitely generated algebra over $k$. 
We say that $R$ is a \textit{rational singularity} if there exists a resolution of singularities $\pi: Y \to\Spec R$, a proper birational morphism with $Y$ a regular scheme, such that $R^i\pi_*\sO_{Y}=0$ for every $i>0$. 
If $R$ is a rational singularity, then $R^i\pi_*\sO_{Y}=0$ for every $i>0$ for every resolution of singularities $\pi:Y \to\Spec R$. 

It is an application of Grauert-Riemenschneider vanishing Theorem that $R$ is a rational singularity if and only if $R$ is Cohen-Macaulay and $\pi_*\omega_Y=\omega_R$ for some (every) resolution $\pi:Y \to \Spec R$, where $\omega_Y$ is the canonical sheaf on $Y$.  
\end{defn}

Let us introduce the notion of pure subrings. 
Let $A \subset B$ be a ring extension. 
We say that $A$ is a \textit{pure subring} of $B$ if for every $A$-module $M$,the natural map 
$M=M\otimes_A A\to  M\otimes_A B$ is injective.\footnote{For an ideal $I$ of $A$,  the natural map $A/I \to (A/I) \otimes_A B$ is injective if and only if $IB \cap A=I$. 
Under the mild assumption that $A$ is locally excellent and reduced,  the condition that $IB \cap A=I$ for every ideal $I$ in $A$ is equivalent to the one that $A$ is a pure subring of $B$ (\cite{Ho}).}
For example, if $A$ is a direct summand of $B$ as an $A$-module, then $A$ is a pure subring of $B$. 

The following Boutot's theorem is very important in the study of singularities in characteristic zero. 
For example, it is a direct consequence of this theorem that in characteristic zero, rings of invariants under linearly reductive group actions are rational singularities. 

\begin{thm}[Boutot's  Theorem \cite{Bo}] Let 
$A \subset B$ be a  extension of rings of essentially of finite type over a field of characteristic zero, 
 and assume that $A$ is a pure subring of $B$.  
If $B$ is a rational singularity, then so is $A$. 
\end{thm} 

Although Boutot's theorem is proved by using Grauert-Riemenschneider vanishing theorem, the Boutot-type theorem for $F$-regular rings immediately follows from the definition. 

\begin{thm}[Boutot-type theorem for $F$-regular rings]\label{BoF-reg}
Let  $R\subset S$ be an extension of reduced rings of characteristic $p>0$ 
satisfying the condition $R^{\circ} \subseteq S^{\circ}$, and assume that $R$ is a pure subring of $S$. If $S$ is weakly $F$-regular $($respectively, $F$-regular$)$, then so is $R$. 
\end{thm} 

Indeed, it is almost obvious to see that $I^*S \subset (IS)^*$ for every ideal $I$ of $R$. 
If $S$ is weakly $F$-regular, then $I^*\subseteq (IS)^*\cap R=IS\cap R=I$,
which implies that  $R$ is weakly $F$-regular, too. 

\begin{rem}\label{counterexample to Boutot}
While we can think of $F$-rational rings as a positive characteristic analogue of rational singularities, ``Boutot-type theorem for $F$-rational rings" fails to hold. 
Let $k$ be a field of characteristic $p>0$ and let $R= \bigoplus_{n\ge 0} R_n$ be a Noetherian graded ring with $R_0=k$.
It is known in \cite{Wa3} (see also \cite{HWY}) that if $R$ is Cohen-Macaulay isolated singularity and 
if $a(R)<0$ where $a(R)$ is the $a$-invariant of $R$,\footnote{The $a$-invariant $a(A)$ of a $d$-dimensional Noetherian graded ring $(A, \m)=\bigoplus_{n \ge 0}A_n$ with $A_0$ a field, introduced by Goto and Watanabe \cite{GW}, is the largest integer $a$ such that $[H^d_{\m}(A)]_a \ne 0$.}
then $R$ is a pure subring of some $F$-rational graded ring. 
For example, suppose that $R = k[X, Y,Z]/(X^2+Y^3+Z^5)$, and $x, y, z$ denote the images of $X, Y, Z$ in $R$, respectively. 
Then $R$ can be viewed as a graded ring by putting $\deg x=15, \deg y=10, \deg z=6$, respectively, and $a(R)=-1$. 
If $p \le 5$, then $R$ is \textit{not} an $F$-rational ring but a pure subring of some $F$-rational graded ring.
This is a counterexample to ``Boutot-type theorem for $F$-rational rings."

On the other hand, ``Boutot-type Theorem for $F$-rational rings" holds if $R\subset S$ is a finite extension. 
Indeed, if $S$ is a finite $R$-module,  then the extension $IS$ of a parameter ideal $I$ of $R$ is a parameter ideal of $S$. 
Thus, if $S$ is $F$-rational and if $R$ is a pure subring of $S$, then $R$ is $F$-rational by the same argument as the proof of Theorem \ref{BoF-reg}. 
\end{rem}

In order to state Brian\c con-Skoda theorem,  we recall the definition of the integral closure of an ideal.   

\begin{defn} Let $I$ be an ideal of a ring $A$. 
\begin{enumerate}
\item
We say that an element $x\in A$ is \textit{integral over} $I$ if there exist some $n \in \N$ and $a_i \in I^i$ for $i=1, \dots, n$ such that $x^n+a_1x^{n-1}+a_2x^{n-2}+\cdots+a_n=0$. 
This condition is equivalent to saying that there exists some $c \in A^{\circ}$ such that $cx^n \in I^n$ for all sufficiently large $n \in \N$ (\cite[Corollary 6.8.12]{HS}).
\item
 The set of elements of $A$ integral over $I$ is called the \textit{integral closure} of $I$ and denoted by $\overline{I}$. We say that $I$ is \textit{integrally closed} if $I=\overline{I}$. 
\end{enumerate}
\end{defn}

By definition, the tight closure $I^*$ of an ideal $I$ of $R$ is contained in the integral closure $\overline{I}$ of $I$. 
The name of tight closure comes from this fact. 
It is natural to ask how small the tight closure $I^*$ is when compared with the integral closure $\overline{I}$.  

Brian\c con-Skoda theorem was originally proved for convergent power series rings over  
the field of complex numbers $\C$. 
We now formulate an analogous statement for an arbitrary ring of positive characteristic, using tight closure.
This tight closure version of Brian\c con-Skoda theorem gives a simple alternative proof of the original Brian\c con-Skoda theorem by using reduction to positive characteristic (see \S \ref{reduction} for the technique of reduction to positive characteristic).

\begin{thm}[Tight closure version of Brian\c con-Skoda theorem]\label{BS thm}
Let $I$ be an ideal of $R$ generated by $n$ elements. Then for every 
$l \in \N$, we have  $\overline{ I^{n+l-1}}\subset  (I^l)^*$. 
\end{thm}

We show this theorem in the case where $l=1$ for simplicity. 
For $x \in \overline{I^n}$, there exists a $c \in R^{\circ}$ such that $cx^m \in (I^n)^m=I^{mn}$ for all sufficiently large $m \in \N$. 
Since $I$ is generated by $n$ elements, $I^{(q-1)n+1} \subseteq I^{[q]}$ for every power $q$ of $p$.  
If $m=q$ is a power of $p$, then $cx^q \in I^{qn}\subseteq I^{[q]}$. 
Thus, $x \in I^*$.

Before stating a corollary of Theorem \ref{BS thm}, note that a reduced local ring $R$ is normal if and only if the principal ideal $(a)$ is integrally closed for every $a \in R^{\circ}$.   
Considering the case when $n=l=1$ in Theorem \ref{BS thm}, we obtain the following result.  
\begin{cor}\label{Frational=>normal}
If $I $ is a principal ideal of $R$, we have $I^* =\overline{I}$.  In particular, if $R$ is $F$-rational, then $R$ is normal. 
\end{cor}

In order to show that $F$-rational sings are Cohen-Macaulay, we need the following ``colon capturing" property of tight closure. 

\begin{thm}\label{colon capturing} 
Let $(R,\fkm)$ be a $d$-dimensional  equidimensional local ring,  and suppose that $R$ is a homomorphic image of a Cohen-Macaulay local ring. 
Let $x_1,\ldots , x_d$ be a full system of parameters for $R$. Then for every $i=0, 1, \dots, d-1$, one has 
\[
(x_1,\ldots , x_i) : x_{i+1} \subset (x_1,\ldots , x_i)^*, 
\]
where the ideal $(x_1, \dots, x_i)$ is regarded as the zero ideal when $i=0$. 
\end{thm}

Since $F$-rational rings are normal by Corollary \ref{Frational=>normal}, $F$-rational local rings are integral domains. 
Also, note that excellent local rings are homomorphic images of Cohen-Macaulay local rings. 
Therefore, we have the following corollary. 

\begin{cor}
If $(R,\fkm)$ is an excellent $F$-rational local ring, then $R$ is Cohen-Macaulay. 
In particular, every $F$-rational ring essentially of finite type over a field is Cohen-Macaulay.  
\end{cor}

We have discussed the behavior of tight closure of ideals. 
Now we turn our attention to the tight closure of modules. 
Given an $R$-module $M$ and its submodule $N$,  we can define the tight closure $N^*_M$ of $N$ in $M$. 
In this article, we only consider the case where $N=0$ for simplicity.\footnote{Since there exists an isomorphism $N^*_M/N \cong 0^*_{M/N}$, we can reduce to the case where $N=0$.}

In order to state the definition of the tight closure of modules, we need to introduce some notation. 
Given an $R$-module $M$ and an $e \in \N$, the $R$-module $F^e_*M$ is defined by the following two conditions: (1) $F^e_*M=M$ as an abelian group, (2) the $R$-module structure of $F^e_*M$ is given by $r \cdot x:=r^{p^e}x$ with $r \in R$ and $x \in F^e_*M$. 
We write elements of $F^e_*M$ in the form $F^e_*x$ with $x \in M$. 
If $X=\Spec R$ and $F:X \to X$ is the (absolute) Frobenius morphism on $X$, 
then $F^e_*R$ is the $R$-module corresponding to $F^e_*\sO_X$. 
By definition, the $e$-times iterated Frobenius map $F^e: R \to F^e_*R$ sending $x$ to $F^e_*(x^{p^e})=x \cdot F^e_*1$ is  an $R$-module homomorphism. 

\begin{defn}\label{tc module}
Let $M$ be a (not necessarily finitely generated) $R$-module. For every $e \in \N$, 
let $F^e:R \to F^e_*R$ be the $e$-times iterated Frobenius map. 
$F^e$ induces the $e$-th Frobenius map on $M$  
\[
F^e_M: M \to M \otimes _R F^e_*R \quad x \mapsto x \otimes F^e_*1.
\]
Then the tight closure $0_M^*$ of the zero submodule in $M$ is defined as follows: 
an element $x \in M$ belongs to $0_M^*$ if and only if there exists some $c\in R^{\circ}$ such that $cF^e_M(x)=0$ for all sufficiently large $e \in \N$.\footnote{For an ideal $I$ of $R$, if we put $M=R/I$, then the tight closure of $0$ in $M$ agrees with $I^*/I$.}
\end{defn}

Weak $F$-regularity and $F$-rationality can be characterized in terms of $0^*_M$. 
If $0^*_M = 0$ for every finitely generated $R$-module $M$, then $R$ is weakly $F$-regular, because $0^*_{R/I}=I^*/I \subseteq R/I$. 
The converse also holds under the mild assumption that $R$ is locally excellent. 
A $d$-dimensional excellent local ring $(R, \m)$ is $F$-rational if and only if $R$ is Cohen-Macaulay and $0^*_{H^d_{\m}(R)}=0$ (see Proposition \ref{Frational2}). 
Also, an $F$-finite local ring $(R, \m)$ is strongly $F$-regular if and only if $0^*_{E_R(R/\m)}=0$ where $E_R(R/\m)$ is the injective hull of the residue field $R/\m$ (see the definitions of $F$-finiteness and of strongly $F$-regular rings for \S \ref{strongly F-regular section}). 

The Frobenius map $F: R \to R$ sending $r$ to $r^p$ induces a $p$-linear map $H^d_{\m}(R) \to H^d_{\m}(R)$.\footnote{For an $R$-module $M$, we say that a map $\varphi:M \to M$ is \textit{$p$-linear} if $\varphi$ is additive and $\varphi(rz)=r^p\varphi(z)$ with $r \in R$ and $z \in M$.}
 We denote this $p$-linear map by the same letter $F$ if it does not cause any confusion. 
 Karen Smith gave a characterization of $F$-rational rings using this Frobenius action $F$ on  $H^d_{\m}(R)$. 

\begin{thm}[\cite{Sm1}]\label{Frobenius action}
Let $(R, \m)$ be a $d$-dimensional excellent local ring of characteristic $p>0$. Then 
$R$ is $F$-rational if and only if $R$ is Cohen-Macaulay and $H^d_{\fkm} (R)$ has no proper nontrivial submodules stable under the Frobenius action $F$.
\end{thm}

Lipman and Teissier \cite{LT} introduced the notion of pseudo-rational rings as a resolution-free  and characteristic-free analogue of rational singularities. 
Let $A$ be a homomorphic image of an excellent Gorenstein local ring. 
We say that $A$ is \textit{pseudo-rational} if $A$ is normal Cohen-Macaulay and if $\pi_*\omega_Y=\omega_X$ for every proper birational morphism $\pi:Y \to X=\Spec A$ with $Y$ normal. 
A non-local ring is pseudo-rational if all of its localizations at maximal ideals are pseudo-rational. 
It is well-known that pseudo-rationality is equivalent to rational singularities for rings essentially of finite type over a field of characteristic zero.

As a corollary of Theorem \ref{Frobenius action}, Smith proved that $F$-rationality implies pseudo-rationality. 

\begin{cor}[\cite{Sm1}]\label{Smith's thm}
\label{pseudo-rational}
Excellent $F$-rational rings are pseudo-rational.  
\end{cor}
In \S \ref{trace section}, we will explain that there is a more geometric way to show this result in the case where $R$ is $F$-finite. 

\section{Classical $F$-singularities}\label{classical Fsing}

$F$-regular and $F$-rational rings have origin in the theory of tight closure as we have seen in the previous section, but they are nowadays recognized as one of the classes of $F$-singularities.  
``$F$-singularities" are a generic term used to refer to singularities defined in terms of Frobenius maps. 
In addition to $F$-regular and $F$-rational rings, there are two other basic classes of $F$-singularities, $F$-pure and $F$-injective rings. 
In this section, we will overview these 4 classes of singularities from a more algebro-geometric point of view than in \S\ref{tight closure}. 
Roughly speaking, they are divided into two groups, singularities defined via splittings of Frobenius maps and singularities defined via surjectivity of trace maps. 

Let $R$ be a ring of prime characteristic $p$. 
We say that $R$ is \textit{$F$-finite} if $F_*R$ is a finitely generated $R$-module (see the paragraph preceding Definition \ref{tc module} for the definition of $F_*R$). 
A field $K$ is $F$-finite if and only if the extension degree $[K:K^p]$ is finite. 
Important examples of $F$-finite rings are rings essentially of finite type over an $F$-finite field and  complete local rings with $F$-finite residue field.
$F$-finite rings satisfy the following nice property. 

\begin{fact}[\cite{Ku}, \cite{Ga}]
$F$-finite rings are excellent and have a dualizing complex. 
\end{fact}

\subsection{Singularities defined in terms of Frobenius splitting}\label{strongly F-regular section}

Suppose that $R$ is a ring of prime characteristic $p$. 
For each $e \in \N$, the $e$-times iterated Frobenius map $F^e:R \to F^e_*R$ is the $R$-module homomorphism sending $x$ to $F^e_*x^{p^e}=x \cdot F^e_*1$. 
We make the following easy remark, which may help the reader understand various definitions in this section: when $R$ is reduced with minimal prime ideals $\fkp_1, \dots, \fkp_r$, 
the map $F^e$ can be identified with the natural inclusion 
\[R \hookrightarrow \prod_{i=1}^r R/\fkp_i \hookrightarrow R^{1/p^e}=\left\{x \in \prod_{i=1}^r \overline{Q(R/\fkp_i)} \; \Bigg| \; x^{p^e} \in R \right\},\] 
where $\overline{Q(R/\fkp_i)}$ is the algebraic closure of the quotient field $Q(R/\fkp_i)$ of $R/\fkp_i$. 

\begin{defn}[\cite{HR}, \cite{HH1}]\label{strongly F-regular rings}
Let $R$ be an $F$-finite ring of prime characteristic $p$. 
We say that $R$ is \textit{$F$-pure} if the Frobenius map $F: R \to F_*R$ splits as an $R$-module homomorphism. 
We say that $R$ is \textit{strongly $F$-regular} if for every $c \in R^{\circ}$, there exists some $e \in \N$ such that the map 
\[
cF^e: R \xrightarrow{F^e} F^e_*R \xrightarrow{\times F^e_*c} F^e_*R \quad x \mapsto F^e_*(x^{p^e}) \mapsto  F^e_*(cx^{p^e})
\]
splits as an $R$-module homomorphism. 
\end{defn}

Since the splitting of $F^e:R \to F^e_*R$ for some $e \in \N$ implies that of $F:R \to F_*R$, we see that strongly $F$-regular rings are $F$-pure by considering the case when $c=1$.
Note also that if the map $F:R \to F_*R$ splits, then it has to be injective. Thus, $F$-pure rings are reduced.  

\begin{ex}
Let $R=k[x_1, \dots, x_d]$ be a polynomial ring over a perfect field $k$ of characteristic $p>0$. 
Then $R$ is strongly $F$-regular. 
Indeed, let $c$ be any nonzero polynomial of $R$. 
Take a sufficiently large $e \in \N$ such that the degree of $c$ in $x_j$ is less than $p^e$ for all $j=1, \dots, d$. 
In particular, $c$ has a term $a x_1^{m_1} \cdots x_d^{m_d}$ with $m_j <p^e$ and $a \in k^*$. 
Since $F^e_*R=\bigoplus_{0 \le i_1, \dots, i_d < p^e} R \cdot F^e_*(x_1^{i_1} \cdots x_d^{i_d})$, 
there exists an $R$-module homomorphism $\varphi:F^e_*R \to R$ such that 
\[
\varphi(F^e_*(x_1^{i_1} \cdots x_d^{i_d}))=\left\{
\begin{array}{ll}
1 & \textup{if $(i_1, \dots, i_d)=(m_1, \dots, m_d)$} \\
0 & \textup{if $(i_1, \dots, i_d) \ne (m_1, \dots, m_d)$ and $0 \le i_1, \dots, i_d < p^e$}.
\end{array}
\right.
\]
Then $\varphi(F^e_*c)=a^{1/p^e}$ and $a^{-1/p^e}\varphi$ gives a splitting of $cF^e:R \to F^e_*R$. 
\end{ex}

More generally, we can show the following proposition, which is an easy consequence of Theorem \ref{Kunz}. 

\begin{prop}[\cite{HH1}]\label{regular => F-regular}
$F$-finite regular rings are strongly $F$-regular. 
\end{prop}

\begin{rem}\label{F-regular remark}
(1) 
Strong $F$-regularity implies $F$-regularity defined in Definition \ref{Frational}. 
Indeed, since strong $F$-regularity commutes with localization, it suffices to show that a strongly $F$-regular ring $R$ is weakly $F$-regular. 
Let $I$ be an arbitrary ideal in $R$, and fix an $x \in I^*$. We will show that $x \in I$. 
By definition, there is a $c \in R^{\circ}$ such that $cx^{p^e} \in I^{[p^e]}$ for all sufficiently large $e \in \N$. 
By the assumption on $R$, there exist an $e \in \N$, which can be made sufficiently large, and an $R$-module homomorphism $\psi:F^e_*R \to R$ sending $F^e_*c$ to $1$. 
Therefore, 
\[x=x\psi(F^e_*c)=\psi(F^e_*(cx^{p^e})) \in \psi(F^e_*(I^{[p^e]}))=I \psi(F^e_*R) \subset I.\] 
 
Conversely, it is conjectured that an $F$-finite weakly $F$-regular ring $R$ is strongly $F$-regular.  This conjecture is known to be true when $R$ is $\Q$-Gorenstein (see \cite{AM})\footnote{When $R$ is local, the conjecture holds if the non-$\Q$-Gorenstein locus of $\Spec R$ is isolated (cf.~\cite{LS2}).} or $R$ is an $\N$-graded ring (see \cite{LS1}).  

(2) It is not hard to see that an $F$-finite ring $R$ is strongly $F$-regular (respectively, $F$-pure) if and only if so is the local ring $R_{\fkp}$ for every maximal ideal $\fkp$ in $R$. 
In particular, if $\m$ is a maximal ideal of $R$ and $R_{\fkp}$ is regular for all other maximal ideals $\fkp \ne \m$, then $R$ is strongly $F$-regular if and only if $R_{\m}$ is strongly $F$-regular. 
\end{rem}

It is usually difficult to determine directly from the definition whether a given ring $R$ is strongly $F$-regular or not, because we have to check the condition for all $c \in R^{\circ}$. 
Pick an element $c \in R^{\circ}$ such that the localization $R_c$ is regular. 
It then follows from Proposition \ref{regular => F-regular} and Lemma \ref{single c} that $R$ is strongly $F$-regular if (and only if) there is an $e \in \N$ such that $cF^e:R \to F^e_*R$ splits as an $R$-module homomorphism. 
Namely, it suffices to check the condition only for this $c$. 

\begin{lem}[\cite{HH1}]\label{single c}
Let $R$ be an $F$-finite reduced ring of characteristic $p>0$ and $c \in R^{\circ}$ an element such that the localization $R_c$ is strongly $F$-regular. 
If there exists an $e \in \N$ such that $cF^e:R \to F^e_*R$ splits as an $R$-module homomorphism, then $R$ is strongly $F$-regular. 
\end{lem}

The following proposition, so-called Fedder's criterion, is a very useful criterion for strong $F$-regularity and $F$-purity. 

\begin{prop}[\cite{Fe}, \cite{Gl}]\label{Fedder}
Let $(S, \n)$ be an $F$-finite regular local ring and $I$ a radical ideal of $S$. 
Set $R=S/I$ and let $c \in S \setminus I$ be an element such that the localization $R_c$ is strongly $F$-regular. 
\begin{enumerate}
\item 
$R$ is $F$-pure if and only if $(I^{[p]}:I) \not\subset \n^{[p]}$. 
\item 
$R$ is strongly $R$-regular if and only if there exists some $e \in \N$ such that $c(I^{[p^e]}:I) \not\subset \n^{[p^e]}$. 
\end{enumerate}
\end{prop}

It is easier to verify the conditions in Proposition \ref{Fedder} when $R$ is a hypersurface, 
because if $I=(f)$ is a principal ideal, then $(I^{[p^e]}:I)=(f^{p^e-1})$. 

\begin{ex}\label{F-pure example}
Let $k$ be a perfect field of characteristic $p>0$. 

(1) Let $R=k[X, Y, Z]/(X^3-YZ(Y+Z))$ and $x, y, z$ be the images of $X, Y, Z$ in $R$, respectively. 
Then we will show that $R$ is $F$-pure if and only if $p \equiv 1$ $\mathrm{mod}$ $3$. 

Since $R$ is singular only at the origin, by Remark \ref{F-regular remark} (2), it suffices to check the $F$-purity of the local ring $R_{\m}$ where $\m=(x, y, z)$. 
Setting $f=X^3-YZ(Y+Z)$, we see from Proposition \ref{Fedder} (1) that $R_{\m}$ is $F$-pure if and only if $f^{p-1} \notin (X^p, Y^p, Z^p)$. 
Note that $f^{p-1} \notin (X^p, Y^p, Z^p)$ if and only if the monomial $X^{p-1}Y^{p-1}Z^{p-1}$ appears in the expansion of $f^{p-1}$. 
If $p \equiv 1 \; \mathrm{mod}\; 3$, then $X^{p-1}Y^{p-1}Z^{p-1}$ appears in the expansion of $f^{p-1}$. 
If $p=3$ or $p \equiv 2 \; \mathrm{mod}\; 3$, then it does not. 
Thus, we obtain the assertion. 

It also follows from Proposition \ref{Fedder} (2) that $R$ is not strongly $F$-regular, because $Xf^{p^e-1} \in (X^{p^e}, Y^{p^e}, Z^{p^e})$ for all $e \in \N$. 

(2) Let $R=k[X, Y, Z]/(X^2+Y^3+Z^5)$, and suppose that $p \ge 7$. 
Then we will show that $R$ is strongly $F$-regular.  

By Remark \ref{F-regular remark} (2), it suffices to check the strong $F$-regularity of $R$ at the origin.  
Set $f=X^2+Y^3+Z^5$. 
By Proposition \ref{Fedder} (2), it is enough to show that either $Yf^{p-1}$ or $Zf^{p-1}$ is not contained in $(X^p, Y^p, Z^p)$. 
If $p \equiv 1\; \mathrm{mod}\; 3$, then the monomial $X^{p-1}Y^{p-1}Z^{(5p+1)/6}$ appears in the expansion of  $Zf^{p-1}$. 
Since $(5p+1)/6 \le p-1$, this monomial is not contained in $(X^p, Y^p, Z^p)$. 
If $p \equiv 2\; \mathrm{mod}\; 3$, then $p \ge 11$ and the monomial  $X^{p-1}Y^{p-1}Z^{(5p+5)/6}$ appears in the expansion of $Yf^{p-1}$.  
Since $(5p+5)/6 \le p-1$, this monomial is not contained in $(X^p, Y^p, Z^p)$. 
Thus, we obtain the assertion. 

\end{ex}

\subsection{Singularities defined in terms of surjectivity of trace maps}\label{trace section}
Let $(R,\m)$ be a $d$-dimensional local ring of prime characteristic $p$. The Frobenius map $F: R \to F_*R$ induces an $R$-linear map between local cohomology modules: 
\[
H^i_{\m}(R)=R \otimes_R H^i_{\m}(R) \to F_*R \otimes_R  H^i_{\m}(R) \cong H^i_{\m}(F_*R).
\]
Under the identification of $F_*R$ with $R$, we view this map as a $p$-linear map $H^i_{\m}(R) \to H^i_{\m}(R)$, and we use the same letter $F$ to denote this $p$-linear map. 
Let $x_1, \dots, x_d$ be a system of parameters for $R$. 
Since $H^{d}_{\m}(R) \cong \varinjlim R/(x_1^n, \dots, x_d^n)$,  the map $F$ on $H^d_{\m}(R)$ can be described as follows: 
\[
F:H^d_{\m}(R) \to H^d_{\m}(R) \quad \xi=[z \; \mathrm{mod} \; (x_1^n, \dots, x_d^n)] \mapsto \xi^p=[z^p \; \mathrm{mod} \; (x_1^{np}, \dots, x_d^{np})].
\]

The following proposition, which one may take as the definition of $F$-rational rings, immediately follows from this description. 
\begin{prop}\label{Frational2}
Let $(R, \m)$ be a $d$-dimensional $F$-finite local ring. 
Then $R$ is $F$-rational $($see Definition $\ref{Frational}$ for the definition of $F$-rational rings$)$ if and only if $R$ is Cohen-Macaulay and if for any $c \in R^{\circ}$, there exists an $e \in \N$ such that 
\[
cF^e:H^d_{\m}(R) \to H^d_{\m}(R) \quad \xi \mapsto c\xi^{p^e}
\]
is injective. 
\end{prop}

\begin{defn}[\cite{Fe}]
A $d$-dimensional $F$-finite local ring $(R, \m)$ is said to be \textit{F-injective} if $F: H^i_{\m}(R) \to H^i_{\m}(R)$ is injective for all $i$. 
When an $F$-finite ring $R$ is not local, we say that $R$ is $F$-injective if the local ring $R_{\mathfrak{p}}$ is $F$-injective for every maximal ideal $\mathfrak{p}$ of $R$. 
\end{defn}

By an argument analogous to the proof that strongly $F$-regular rings are $F$-pure, we see that $F$-rational rings are $F$-injective. 
Moreover, there is the following relationship between strongly $F$-regular (respectively, $F$-pure) rings and $F$-rational (respectively, $F$-injective) rings. 

\begin{prop}\label{F-pure=>F-injective}
$F$-pure $($respectively, strongly $F$-regular$)$ rings are $F$-injective $($respectively, $F$-rational$)$. 
If the ring is $F$-finite and quasi-Gorenstein, then the converse is also true. 
\end{prop}

Proposition \ref{F-pure=>F-injective} follows from Lemma \ref{F-pure lemma}: If a $d$-dimensional $F$-finite local ring $(R, \m)$ is $F$-pure, then by Lemma \ref{F-pure lemma} $(2) \Rightarrow (3)$, $F \otimes \mathrm{id}_{H^i_{\m}(R)}: H^i_{\m}(R) \to F_*R \otimes_R H^i_{\m}(R)$ is injective for each $i$, in other words, $R$ is $F$-injective. 
If $R$ is quasi-Gorenstein, then $E \cong H^d_{\m}(R)$, so that $F$-injectivity implies $F$-purity by Lemma \ref{F-pure lemma} $(1) \Rightarrow (4)$. 
The strongly $F$-regular case is similar: 
As explained in Remark \ref{F-regular remark}, strong $F$-regularity implies weak $F$-regularity and in particular $F$-rationality. 
If $R$ is Gorenstein, then $E \cong H^d_{\m}(R)$, so that $F$-rationality implies the strong $F$-regularity by an analogous statement to Lemma \ref{F-pure lemma} for strongly $F$-regular rings. 

The following lemma is a direct consequence of the definition of $F$-purity and Matlis duality.  
\begin{lem}\label{F-pure lemma}
Let $(R, \m)$ be a $d$-dimensional $F$-finite local ring and $E=H^d_{\m}(\omega_R)$ the injective hull of the residue field $R/\m$. Then the following conditions are equivalent to each other. 
\begin{enumerate}
\item 
$R$ is $F$-pure.
\item
The $R$-dual of the Frobenius map $F^{\vee}: \mathrm{Hom}_R(F_*R, R) \to \mathrm{Hom}_R(R, R)=R$ is surjective. 
\item 
For any $R$-module $M$, $F \otimes \mathrm{id}_M: M=R \otimes_R M \to F_*R \otimes_R M$ is injective. 
\item 
$F \otimes \mathrm{id}_E: E=R \otimes_R E \to F_*R \otimes_R E$ is injective. 
\end{enumerate}
A similar statement holds for strongly $F$-regular rings. 
\end{lem}

Note that $F$-finite $F$-rational rings are normal by Corollary \ref{Frational=>normal}.  
On the other hand, $F$-finite $F$-injective rings are not necessarily normal. 
Before stating the next proposition, we recall the definition of weakly normal rings. 
Since we are working in positive characteristic, we take the following characterization of weak normality given in \cite{RRS} as its definition. 
Let $R$ be an excellent reduced ring of characteristic $p>0$ and $R^{\rm N}$ be the normalization of $R$. We say that $R$ is \textit{weakly normal} if for any $x \in R^{\rm N}$, if $x^p \in R$, then $x \in R$. 

\begin{prop}[\cite{Sch1}]\label{weakly normal}
$F$-finite $F$-injective rings are weakly normal. 
\end{prop}

In the process of proving Proposition \ref{weakly normal}, we obtain the following example. 

\begin{ex}[\cite{Sch1}, \cite{GW2}]
A one-dimensional $F$-finite reduced ring is $F$-injective if and only if it is weakly normal. 
Furthermore, suppose that $R$ is a one-dimensional $F$-finite reduce local ring with perfect residue field. Then $R$ is $F$-pure if and only if it is weakly normal. 
If $R$ is a one-dimensional complete local ring with algebraically closed residue field $k$, then $R$ is isomorphic to $k[[X_1, \dots, X_r]]/(X_iX_j \; | \; i<j)$, where $r$ is the number of associated prime ideals of $R$. 
\end{ex}

Summing up the above, we have seen that the following implications hold for $F$-finite rings: 
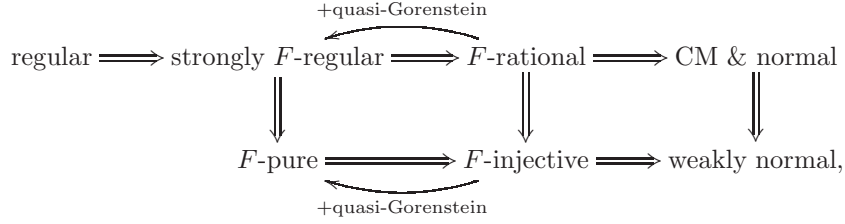
\vspace*{-1.2pt}
\[
\xymatrix{
\textup{regular}\ar@{=>}[r] & 
\textup{strongly $F$-regular} \ar@{=>}[r] \ar@{=>}[d] & \textup{$F$-rational} \ar@{=>}[r] \ar@/_1pc/[l]_{+\textup{quasi-Gorenstein}} \ar@{=>}[d] & \textup{CM \& normal} \ar@{=>}[d] & \\
& \textup{$F$-pure}  \ar@{=>}[r]   & \textup{$F$-injective} \ar@{=>}[r] \ar@/^1pc/[l]^{+\textup{quasi-Gorenstein}} & \textup{weakly normal}, 
}
\]
where ``CM" is an abbreviation for Cohen-Macaulay.

There is no criteria for $F$-rationality such as Proposition \ref{Fedder}, but as we see in the following proposition, we can determine using the $a$-invariant\footnote{See the footnote to Remark \ref{counterexample to Boutot} for the definition of the $a$-invariant.} whether a given graded ring is $F$-rational or not.
This should be compared with a criterion for rational singularities given in \cite{Fl}, \cite{Wa1}. 
\begin{prop}[\cite{HH4}]\label{F-rational criterion}
Let $R=\bigoplus_{n \ge 0}R_n$ be a $d$-dimensional graded ring with $R_0$ a field of characteristic $p>0$ and $\m=\bigoplus_{n \ge 1}R_n$. 
Then $R$ is $F$-rational if and only if the following 4 conditions are satisfied: 
$(1)$ $R$ is Cohen-Macaulay, $(2)$ $R$ is $F$-injective, $(3)$ $R_{\mathfrak{p}}$ is $F$-rational for every prime ideal $\mathfrak{p} \ne \m$, $(4)$ $a(R)<0$ $($the condition $(4)$ can be replaced by the condition that $[H^d_{\m}(R)]_0=0$ in this setting$)$. 
\end{prop}

\begin{ex}
Let $k$ be an algebraically closed field of characteristic $p>0$ with $p \equiv 1 \; \mathrm{mod}\; 3$ and $S=k[X,Y,Z]/(X^3-YZ(Y+Z))$. 
We have seen in Example \ref{F-pure example} that $S$ is $F$-pure. 
Let $\omega \in k$ be a primitive cube root of unity, and suppose that the cyclic group $G=\Z/3\Z$ of order 3 acts on $S$ as $x \mapsto x, y \mapsto \omega y, z \mapsto \omega z$, where $x, y, z$ are the images of $X, Y, Z$ in $S$, respectively. 
Then the invariant subring 
\[
R:=S^G=k[X, Y^3, Y^2Z, Z^3]/(X^3-YZ(Y+Z))
\]
is $F$-rational. 
Indeed, since $S$ is a two-dimensional normal ring, so is $R$. In particular, $R$ is a Cohen-Macaulay isolated singularity. Hence, the conditions (1) and (3) in Proposition \ref{F-rational criterion} are satisfied. 
Since $S$ is $F$-pure, $R$ is also $F$-pure by Lemma \ref{pure subring} (1), which implies the condition (2). 
In order to show that $R$ satisfies the condition (4), note that the third Veronese subring $R^{(3)}$ of $R$ is isomorphic to the third Veronese subring $A^{(3)}=k[Y^3, Y^2Z, YZ^2, Z^3]$ of the polynomial ring $A=k[Y, Z]$. 
Denoting by $\m$ the unique homogeneous maximal ideal of $R$, one has  
\[[H^2_{\m}(R)]_0=[H^2_{\m}(R^{(3)})]_0 \cong [H^2_{(Y,Z)}(A^{(3)})]_0=[H^2_{(Y,Z)}(A)]_0=0.\]
Thus, it follows from Proposition \ref{F-rational criterion} that $R$ is $F$-rational. 
On the other hand, $R$ is not strongly $F$-regular by Lemma \ref{pure subring} (2), because $S$ is not strongly $F$-regular as we have seen in Proposition \ref{F-rational criterion}. 
\end{ex}

\begin{lem}\label{pure subring}
Let $A \subset B$ be an extension of $F$-finite rings such that $A^{\circ} \subseteq B^{\circ}$, and suppose that $A$ is a pure subring of $B$. 
\begin{enumerate}
\item $($\cite{Has}$)$
If $B$ is $F$-pure $($respectively, strongly $F$-regular$)$, then so is $A$. 
\item $($\cite{Wa2}$)$
Suppose that $A$ and $B$ are both normal and that the extension $A \subset B$ is \'{e}tale in codimension one.\footnote{An analogous statement holds even when $A \subset B$ is not \'{e}tale in codimension one. The reader is referred to \cite{ScTu} for the details.}
Then $A$ is $F$-pure $($respectively, strongly $F$-regular$)$ if and only if so is $B$. 
\end{enumerate}
\end{lem}

\begin{ex}\label{F-injective example}
Let $S=k[X, Y, Z]/(X^4+Y^4+Z^4)$ where $k$ is a perfect field of characteristic $p>0$ with $p \equiv 1 \; \mathrm{mod} \; 4$. 
 It is easy to see from Proposition \ref{Fedder} that $S$ is not $F$-pure. 
Therefore, the second Veronese subring 
\[R:=S^{(2)}=k[X^2, XY, XZ, Y^2, YZ, Z^2]/(X^4+Y^4+Z^4)\] 
of $S$ is also not $F$-pure by Lemma \ref{pure subring} (2). 
We will show that $R$ is $F$-injective. 

Let $A=k[X, Y ,Z]$ and $f=X^4+Y^4+Z^4 \in A$. 
Let $\m_A, \m_R$ and $\m_S$ denote the unique homogeneous maximal ideals of $A$, $R$ and $S$, respectively. 
Consider a commutative diagram with exact rows 
\[
\xymatrix{
0 \ar[r] & A \ar[r]^{\times f}  \ar[d]_{f^{p-1}F} & A  \ar[r] \ar[d]_F & S  \ar[r] \ar[d]_F & 0 \\
0 \ar[r] & A \ar[r]^{\times f} & A  \ar[r] & S \ar[r] & 0. 
}
\]
This diagram induces the following commutative diagram with exact rows: 
\[
\xymatrix{
0 \ar[r] & H^2_{\m_S}(S) \ar[r] \ar[d]_F & H^3_{\m_A}(A) \ar[r]^{\times f}  \ar[d]_{f^{p-1}F} & H^3_{\m_A}(A) \ar[d]_F\\
0 \ar[r] & H^2_{\m_S}(S) \ar[r] & H^3_{\m_A}(A)  \ar[r]^{\times f} & H^3_{\m_A}(A). 
}
\]
By the commutativity of this diagram, the Frobenius map $F:H^2_{\m_S}(S) \to H^2_{\m_S}(S)$ on $H^2_{\m_S}(S)$ can be identified with $f^{p-1}F: (0:f)_{H^3_{\m_A}(A)} \to (0:f)_{H^3_{\m_A}(A)}$. 

Suppose that $\xi=[g/(XYZ)^m] \in H^3_{\m_A}(A)$ is a homogeneous element such that $f^{p-1}F(\xi)=0$, that is, $f^{p-1}g^p \in (X^{mp}, Y^{mp}, Z^{mp})$. 
Since $p \equiv 1 \; \mathrm{mod} \; 4$, the monomial $X^{2(p-1)}Y^{p-1}Z^{p-1}=X^p(X^{p-2}Y^{p-1}Z^{p-1})$ appears with a nonzero coefficient $c \in k$ in the expansion of $f^{p-1}$. 
Let $\varphi:F_*A \to A$ be the $A$-linear map sending $F_*(X^{p-2}Y^{p-1}Z^{p-1})$ to $c^{-1/p}$ and the other members of the free basis to zero. 
Then
\[
Xg=\varphi(F_*f^{p-1})g=\varphi(F_*(f^{p-1}g^p)) \in \varphi(F_*(X^{mp}, Y^{mp}, Z^{mp})) \subset (X^m, Y^m, Z^m).
\] 
One can show similarly that $Yg$ and $Zg$ lie in $(X^m, Y^m, Z^m)$. 
Namely, $\m_Ag$ is contained in $(X^m, Y^m, Z^m)$, which is equivalent to saying that $\m_A \xi=0$. 
Since $A$ is a three-dimensional polynomial ring, this means that $\deg \xi \ge -3$. 

Taking into consideration the fact that $[H^2_{\m_S}(S)]_n \cong [(0:f)_{H^3_{\m_A}(A)}]_{n-4}$ for each $n \in \Z$, one has that $F:H^2_{\m_S}(S) \to H^2_{\m_S}(S)$ is injective in nonpositive degrees, and then so is $F:H^2_{\m_R}(R) \to H^2_{\m_R}(R)$.  
Since $[H^2_{\m_R}(R)]_n \cong [H^2_{\m_S}(S)]_{2n}=0$ for all $n \ge 1$, we can conclude that the Frobenius map $F:H^2_{\m_R}(R) \to H^2_{\m_R}(R)$ is injective in all degrees. 
Thus, $R$ is $F$-injective. 
\end{ex}

Let $R$ be an $F$-finite local ring. 
The \textit{trace map} $\mathrm{Tr}_F:F_*\omega_R \to \omega_R$ of the Frobenius map on $R$ is the $\omega_R$-dual of the Frobenius map $F: R \to F_*R$: 
\[
\mathrm{Tr}_F:F_*\omega_R \cong \Hom_R(F_*R, \omega_R) \to \Hom_R(R, \omega_R)=\omega_R, 
\]
where the first isomorphism follows from Grothendieck-Serre duality. 
For each $e \in \N$, the map $\mathrm{Tr}_F^e:F^e_*\omega_R \to \omega_R$ is defined by $\mathrm{Tr}_F^e=\mathrm{Tr}_F \circ F_*\mathrm{Tr}_F \circ \cdots \circ F^{e-1}_*\mathrm{Tr}_F$, which is nothing but the $\omega_R$-dual of the $e$-times iterated Frobenius $F^e:R \to F^e_*R$. 

When $R$ is Cohen-Macaulay, $F$-injectivity and $F$-rationality are characterized in terms of the subjectivity of the trace of Frobenius.  
This is an easy consequence of local duality. 

\begin{prop}
Let $R$ be an $F$-finite reduced local ring. 
\begin{enumerate}
\item 
If $R$ is $F$-injective, then $\mathrm{Tr}_F: F_* \omega_R \to \omega_R$ is surjective. 
When $R$ is Cohen-Macaulay, the converse is also true. 
\item 
$R$ is $F$-rational if and only if $R$ is Cohen-Macaulay and if for any $c \in R^{\circ}$, there exists an $e \in \N$ such that 
\[
\hspace*{1em} c\mathrm{Tr}^e_F:F^e_*\omega_R \xrightarrow{\times F^e_*c} F^e_*\omega_R \xrightarrow{\mathrm{Tr}_F^e} \omega_R \quad F^e_*x \mapsto F^e_*(cx) \mapsto \mathrm{Tr}_F^e(F^e_*(cx))
\]
is surjective. 
\end{enumerate}
\end{prop}
 
A morphism $\pi:Y \to X$ of integral schemes is called an \textit{alteration} if $\pi$ is a generically finite proper surjective morphism. 
It is called a \textit{regular alteration} if $Y$ is a regular scheme. 
It follows from Stein factorization that an alteration $\pi: Y \to X$ factors as $Y \xrightarrow{f} Z \xrightarrow{g} X$, where $f$ is a proper birational morphism and $g$ is a finite morphism. 
Then the trace map $\mathrm{Tr}_{\pi}:\pi_*\omega_Y \to \omega_X$ of $\pi$ is the composite of $\mathrm{Tr}_f$ and $\mathrm{Tr}_g$, where 
the trace map $\mathrm{Tr}_f: f_*\omega_Y \to \omega_Z$ of $f$ is a natural inclusion and the trace map $\mathrm{Tr}_g:g_*\omega_Z \to \omega_X$ of $g$ is the $\omega_X$-dual of $\sO_X \to g_*\sO_Z$: 
\[
g_*\omega_Z \cong \mathcal{H}\mathrm{om}_{\sO_X}(g_*\sO_Z, \omega_X) \to \mathcal{H}\mathrm{om}_{\sO_X}(\sO_X, \omega_X)=\omega_X. 
\]
 
Rational singularities in characteristic zero and $F$-rational rings in characteristic $p>0$ have a common characterization in terms of the surjectivity of the trace map $\mathrm{Tr}_{\pi}: \pi_*\omega_Y \to \omega_X$ for every regular alterations $\pi: Y \to X$.

\begin{thm}[\cite{BST}]\label{BST}
Let $R$ be an $F$-finite Cohen-Macaulay domain. Then the following conditions are equivalent to each other. 
\begin{enumerate}
\item $R$ is $F$-rational. 
\item For any finite extension $R \subset S$, its trace map $\omega_S \to \omega_R$ is surjective. 
\item For any alteration $\pi: Y \to X=\Spec R$, the trace map $\mathrm{Tr}_{\pi}:\pi_*\omega_Y \to \omega_X$ of $\pi$ is surjective. 
\end{enumerate}
If $R$ is of finite type over a perfect field, then we may take $\pi$ to be a regular separable alteration in $(3)$.  
\end{thm}

Under the assumption of $F$-finiteness, Corollary \ref{Smith's thm} immediately follows from Theorem \ref{BST} $(1) \Rightarrow (3)$. 

\subsection{Reduction from characteristic zero to positive characteristic}\label{reduction}
We can define the notion of $F$-singularities in characteristic zero, using reduction from characteristic zero to positive characteristic.  
First we briefly review how to reduce things from characteristic zero to characteristic $p > 0$. Our main references are \cite[Chapter 2]{HH3} and \cite[Section 3.2]{MS}.

Let $R$ be a ring of finite type over a field $k$ of characteristic zero. 
There exist finitely many polynomials $f_i=\sum_{\underline{j}}a_{i\underline{j}}\underline{X}^{\underline{j}} \in k[X_1, \dots, X_n]$ such that $R \cong k[X_1, \dots, X_n]/(f_1, \dots, f_r)$. 
Let $A=\Z[a_{i\underline{j}}] \subseteq k$ be the $\Z$-subalgebra generated by the coefficients of the $f_i$ and $R_A=A[X_1, \dots, X_n]/(f_1, \dots, f_r)$. 
Then $R_A \otimes_A k \cong R$. 
Applying the generic freeness (\cite[Theorem 24.1]{Ma}), after possibly replacing $A$ by the localization $A_a$ with respect to some element $a \in A$, we may assume that $R_A$ is flat over $A$. 
Such a ring $R_A$ is referred to as a \textit{model} of $R$ over $A$. 
For a closed point $\mu \in \Spec A$, denote by $R_{\mu}$ the closed fiber $R_A \otimes_A A/\mu$ of the natural map $A \to R_A$ over $\mu$. 
Note that $R_{\mu}$ is $F$-finite, because $R_A$ is of finite type over the finite field $A/\mu$. 
If $R$ is regular (respectively, $\Q$-Gorenstein, Cohen-Macaulay, normal), after possibly replacing $A$ by a localization $A_a$, we may assume that $R_A$ is regular (respectively, $\Q$-Gorenstein, Cohen-Macaulay, normal), and then so is $R_{\mu}$ for all closed points $\mu \in \Spec A$. 

When $X$ is a scheme of finite type over $k$, we can similarly define a scheme $X_A$ of finite type over a finitely generated $\Z$-subalgebra $A$ of $k$ and a scheme $X_{\mu}$ of finite type over $A/\mu$ for each closed point $\mu \in \Spec A$. 
Given a morphism $f:X \to Y$ of schemes of finite type over $k$ and models $X_A$ and $Y_A$ of $X$ and $Y$ over $A$, respectively, after enlarging $A$ if necessarily, we may assume that $f$ induces a morphism $f_A:X_A \to Y_A$ of schemes of finite type over $A$. 
Then we can define a morphism $f_{\mu}:X_{\mu} \to Y_{\mu}$ of schemes of finite type over $A/\mu$ for each closed point $\mu \in \Spec A$. 
If $f$ is a projective morphism (respectively, a finite morphism), after possibly enlarging $A$ again, we may assume that $f_{\mu}$ is projective (respectively, finite) for all closed points $\mu \in \Spec A$. 

Suppose that $X$ is a normal variety over $k$ and $D=\sum_{i}d_i D_i$ is a $\Q$-divisor on $X$. 
Taking a model $D_{i,A} \subset X_A$ of $D_i$ over $A$ for each $i$, we say that $D_{A}:=\sum_{i}d_i D_{i, A}$ is a model of $D$ over $A$. 
After enlarging $A$ if necessarily, we may assume that $D_{i, \mu}$ is a prime divisor on the normal variety $X_{\mu}$ for every closed point $\mu \in \Spec A$. 
Then $D_{\mu}:=\sum_{i}d_i D_{i, \mu}$ is a $\Q$-divisor on $X_{\mu}$. 
If $D$ is Cartier (respectively, $\Q$-Cartier), then after possibly enlarging $A$ again, we may assume that $D_{\mu}$ is Cartier (respectively, $\Q$-Cartier) for every closed point $\mu \in \Spec A$. 

\begin{defn}
Let $R$ be a ring essentially of finite type over a field $k$ of characteristic zero. 
Suppose we are given a model $R_A$ of $R$ over a finitely generated $\Z$-subalgebra $A$ of $k$. 
We say that $R$ is of \textit{$F$-rational type} (respectively, \textit{strongly $F$-regular type}) if there exists a dense open subset $S \subset \Spec A$ such that $R_{\mu}$ is $F$-rational (respectively, strongly $F$-regular) for all closed points $\mu \in S$. 
We say that $R$ is of \textit{dense $F$-injective type} (respectively, \textit{dense $F$-pure type}) if there exists a dense subset of closed points $S \subset \Spec A$ such that $R_{\mu}$ is $F$-injective (respectively, $F$-pure) for all $\mu \in S$. 
These definitions are independent of the choice of $R_A$. 
\end{defn} 

\begin{ex}
(1) Let $R=\C[X,Y,Z]/(X^3-YZ(Y+Z))$. 
Then the ring $R_{\Z}=\Z[X, Y,Z]/(X^3-YZ(Y+Z))$ is a model of $R$ over $\Z$. 
We take the dense subset $S$ of $\Spec \Z$ to be $\{p \in \Spec \Z \; | \; p \equiv 1 \; \mathrm{mod} \; 3\}$. 
Since we have seen in Example \ref{F-pure example} (1) that $R_p=\F_p[X, Y,Z]/(X^3-YZ(Y+Z))$ is $F$-pure for all $p \in S$, the ring $R$ is of dense $F$-pure type. 

(2) Let $R=\C[X, Y, Z]/(X^2+Y^3+Z^5)$. 
The ring $R_{\Z}=\Z[X, Y,Z]/(X^2+Y^3+Z^5)$ is a model of $R$ over $\Z$. 
Note that $S=\{p \in \Spec \Z \; | \; p \ge 7\} \cup \{0\}$ is a dense open subset of $\Spec \Z$. 
Since $R_p=\F_p[X, Y,Z]/(X^2+Y^3+Z^5)$ is strongly $F$-regular for all closed points $p \in S$ by Example \ref{F-pure example} (2), the ring $R$ is of  strongly $F$-regular type. 
\end{ex}

$F$-singularities (partly conjecturally) correspond to singularities arising in birational geometry in characteristic zero. 
Before explaining this correspondence, we first explain what kind of singularities in characteristic zero are considered. 
The following implications holds for singularities in characteristic zero. 
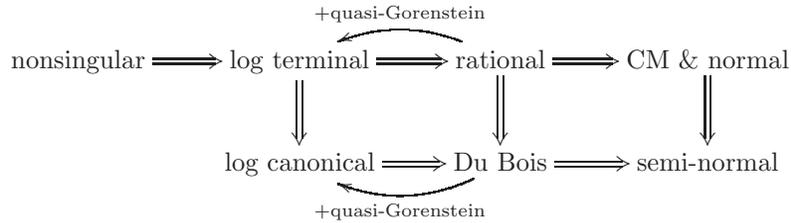
\[
\xymatrix{
\textup{nonsingular} \ar@{=>}[r] & \textup{log terminal} \ar@{=>}[r] \ar@{=>}[d] & \textup{rational} \ar@{=>}[r] \ar@/_1pc/[l]_{+\textup{quasi-Gorenstein}} \ar@{=>}[d] & \textup{CM \& normal} \ar@{=>}[d] \\
& \textup{log canonical}  \ar@{=>}[r]   & \textup{Du Bois} \ar@{=>}[r] \ar@/^1pc/[l]^{+\textup{quasi-Gorenstein}} & \textup{semi-normal}
}
\]
The reader is referred to $\S\ref{MMP sing}$ for the definition of log terminal and log canonical singularities and to \cite{Kov1}, \cite{Kov2}, \cite{KK} for Du Bois singularities. 
The definition of Du Bois singularities are a bit involved, but a simple characterization of them was given in \cite{Sch5} when they are Cohen-Macaulay and normal: 
Let $R$ be a Cohen-Macaulay normal domain essentially of finite type over a field of characteristic zero. 
Let $\pi:\widetilde{X} \to X$ be a log resolution of $X=\Spec R$, that is, $\pi$ is a proper birational morphism with $\widetilde{X}$ nonsingular such that the exceptional locus $E$ of $\pi$ is a simple normal crossing divisor. 
Then $X$ has only \textit{Du Bois singularities} if and only if $\pi_*\omega_{\widetilde{X}}(E)=\omega_X$. 

We also briefly explain the definition of semi-normal rings. 
Let $R$ be an excellent reduced ring and $R^{\rm N}$ be the normalization of $R$. 
We say that $R$ is \textit{semi-normal} if for any $x \in R^{\rm N}$, if $x^2, x^3 \in R$, then $x \in R$. 
In equal characteristic zero, weak normality is equivalent to semi-normality, whereas weak normality is a strictly stronger condition than semi-normality in positive characteristic. 

As the name suggests, there is a correspondence between $F$-rational rings  and rational singularities. 

\begin{thm}[\cite{Ha}, \cite{MeS}]\label{F-rational and rational}
Let $R$ be a ring essentially of finite type over a field of characteristic zero. 
Then $\Spec R$ has only rational singularities if and only if $R$ is of $F$-rational type. 
\end{thm}
Let us say a few words about the proof of Theorem \ref{F-rational and rational}. 
The ``if" part follows from Corollary \ref{pseudo-rational}.  
The following lemma, a consequence of Serre's vanishing theorem and Deligne-Illusie's result \cite{DI} on Akizuki-Kodaira-Nakano's vanishing theorem in characteristic $p>0$, is essential in the proof of  the ``only if" part. 

\begin{lem}[\cite{Ha}]\label{surjectivity}
Let $R$ be a ring essentially of finite type over a field $k$ of characteristic zero. 
Let $\pi:\widetilde{X} \to X$ be a log resolution of $X=\Spec R$ and $E$ be a $\pi$-ample $\Q$-divisor on $X$ whose fractional part $E-\lfloor E \rfloor$ has simple normal crossing support. 
Suppose that we are given models $\pi_A:\widetilde{X}_A \to X_A$ and $E_A$ over a finitely generated $\Z$-subalgebra $A$ of $k$. 
Then there exists a dense open subset $S \subset \Spec A$ such that for every closed point $\mu \in S$ and every $e \in \N$, the map 
\[
F^e_*\pi_{\mu *}\omega_{\widetilde{X}_{\mu}}(\lceil p^e E_{\mu} \rceil) \to \pi_{\mu *}\omega_{\widetilde{X}_{\mu}}(\lceil E_{\mu} \rceil), 
\]
induced by the trace map $\mathrm{Tr}_F^e: F^e_*\omega_{\widetilde{X}_{\mu}} \to \omega_{\widetilde{X}_{\mu}}$, is surjective. 

\end{lem}

\begin{thm}[\cite{Sch1}]\label{Du Bois}
If a ring $R$ essentially of finite type over a field of characteristic zero is of dense $F$-injective type, then $\Spec R$ has only Du Bois singularities. 
\end{thm}

Since $F$-injective rings and Du Bois singularities share many similar properties,  
the converse of Theorem \ref{Du Bois} is also expected to be true. 
However, to the best of our knowledge, Conjecture \ref{DB conj} is open even when $R$ is a two-dimensional normal local ring. 

\begin{conj}\label{DB conj}
Let $R$ be a ring essentially of finite type over a field of characteristic zero. 
Then $\Spec R$ has only Du Bois singularities if and only if $R$ is of dense $F$-injective type.  
\end{conj}

We will explain a correspondence between strongly $F$-regular rings (respectively, $F$-pure rings) and log terminal singularities (respectively, log canonical singularities) in the next section. 


\section{$F$-singularities of pairs}\label{Fpairs}

In the minimal model program, singularities are studied in the pair setting: instead of looking only at singularities of a single variety, one should consider singularities of pairs $(X, \Delta)$ where $X$ is a variety and $\Delta$ is a $\Q$-divisor (that is, a $\Q$-linear combination of divisors) on $X$. 
Since there is a close relationship between $F$-singularities and singularities in birational geometry as we have partly seen in \S\ref{classical Fsing},  we introduce ``$F$-singularities of pairs," that is, a generalization of $F$-singularities to the pair setting. 

\subsection{Correspondence to singularities in the minimal model program}\label{MMP sing}
First we recall the definition of singularities of pairs. 
Let $X$ be a normal variety over a field of characteristic zero and $\Delta=\sum_i d_i \Delta_i$ be an effective $\Q$-divisor on $X$ such that $K_X+\Delta$ is $\Q$-Cartier. 
The round down of $\Delta$ is $\lfloor \Delta \rfloor=\sum_i \lfloor d_i \rfloor \Delta_i$ where $\lfloor d_i \rfloor$ denotes the largest integer less than or equal to $d_i$. 
The round up of $\Delta$ is $\lceil \Delta \rceil=\sum_i \lceil d_i \rceil \Delta_i$ where $\lceil d_i \rceil$ denotes the smallest integer greater than or equal to $d_i$. 

We take a log resolution $\pi:\widetilde{X} \to X$ of $(X, \Delta)$. 
Namely, $\pi$ is a proper birational morphism such that $\widetilde{X}$ is a nonsingular variety and that $\mathrm{Exc}(\pi)$ and $\mathrm{Exc}(\pi) \cup \pi^{-1}_*\Delta$ are simple normal crossing divisors, where  $\mathrm{Exc}(\pi)$ denotes the exceptional locus of $\pi$ and $\pi^{-1}_*\Delta$ does the strict transform of $\Delta$.\footnote{Such a $\pi$ always exists by a famous theorem of Hironaka \cite{Hi}.}
Then we can write
\[
K_{\widetilde{X}}=\pi^*(K_X+\Delta)+\sum_i a_i E_i,
\]
where the $E_i$ are prime divisors on $\widetilde{X}$ and $a_i$ are rational numbers. 
We say that the pair $(X, \Delta)$ is \textit{Kawamata log terminal}  (\textit{klt}, for short) if $a_i>-1$ for all $i$ and that $(X, \Delta)$ is \textit{log canonical} (\textit{lc}, for short) if $a_i \ge -1$ for all $i$. 
Furthermore, suppose that $\pi^{-1}_*\lfloor \Delta \rfloor$ is nonsingular (but possibly disconnected). 
Then we say that $(X, \Delta)$ is \textit{purely log terminal} (\textit{plt}, for short) if $a_i>-1$ for all $i$ with $E_i$ $\pi$-exceptional. 
These definition are independent of the choice of the log resolution $\pi$. 
If $\lfloor \Delta \rfloor=0$, then by definition, being klt is equivalent to being plt. 
When $(X, 0)$ is klt (respectively, lc), we say that $X$ has only log terminal (respectively, log canonical) singularities. 
In general, the following implications hold for singularities of pairs:
\[
\xymatrix{
\textup{klt} \ar@{=>}[r] & \textup{plt} \ar@{=>}[r] & \textup{lc.}}
\]
The reader is referred to \cite{KM} for their basic properties. 

Next, we generalize the definition of $F$-singularities to the pair setting. 
Let $X=\Spec R$ be an $F$-finite normal integral affine scheme and $D$ be an effective integral divisor on $X$.  
We denote by $F:X \to X$ the (absolute) Frobenius morphism on $X$. 
Since $D$ is effective, we have a natural inclusion $i: \sO_X \to \sO_X(D)$. 
For each $e \in \N$, by the composition $F^e_*i$ and the $e$-times iterated Frobenius $F^e$, we have the following map: 
\[
\sO_X \xrightarrow{F^e} F^e_*\sO_X \xrightarrow{F^e_*i} F^e_*\sO_X(D) \quad x \mapsto F^e_*x^{p^e} \mapsto F^e_*x^{p^e}.
\]
We define $F$-singularities of pairs using this map. 
\begin{defn}[\cite{HW}, \cite{Sch2}]
Let $\Delta$ be an effective $\Q$-divisor on $X$. 
\begin{enumerate}
\item We say that $(X, \Delta)$ is \textit{$F$-pure} if the map $\sO_X \to F^e_*\sO_X(\lfloor (p^e-1) \Delta \rfloor)$ splits as an $\sO_X$-module homomorphism for every $e \in \N$. 
\item We say that $(X, \Delta)$ is \textit{sharply $F$-pure} if there exists an $e \in \N$ such that the map $\sO_X \to F^e_*\sO_X(\lceil (p^e-1) \Delta \rceil)$ splits as an $\sO_X$-module homomorphism. 
\item We say that $(X, \Delta)$ is \textit{strongly $F$-regular} if for every nonzero $c \in \sO_X$, there exists an $e \in \N$ such that the composite map 
\[
\hspace*{3em}
\sO_X \to F^e_*\sO_X(\lceil (p^e-1) \Delta \rceil) \xrightarrow{\times F^e_*c} F^e_*\sO_X(\lceil (p^e-1) \Delta \rceil) 
\]
sending $x$ to $F^e_*(cx^{p^e})$ splits as an $\sO_X$-module homomorphism. 
\item We say that $(X, \Delta)$ is \textit{purely $F$-regular}\footnote{Purely $F$-regular pairs are called divisorially $F$-regular pairs in \cite{HW}. Since they correspond to plt pairs, not to dlt pairs, we use the term ``purely $F$-regular" to avoid giving a misleading impression.}
if for every $c \in \sO_X$ which is not in any minimal prime of $\sO_X(-\lfloor \Delta \rfloor)$, 
 there is an $e \in \N$ such that the map 
\[
\hspace*{3em}
\sO_X \to F^e_*\sO_X(\lceil (p^e-1) \Delta \rceil) \xrightarrow{\times F^e_*c} F^e_*\sO_X(\lceil (p^e-1) \Delta \rceil) 
\]
sending $x$ to $F^e_*(cx^{p^e})$ splits as an $\sO_X$-module homomorphism. 
\end{enumerate}
These definitions can be extended to non-affine schemes by considering the same conditions on each affine chart. 
\end{defn}
If $\lfloor \Delta \rfloor=0$ (respectively, $(p^e-1)\Delta$ is an integral divisor for some $e \in \N$), then the pure $F$-regularity (respectively, the sharp $F$-purity) of $(X, \Delta)$ is equivalent to its strong $F$-regularity (respectively, its $F$-purity). 
If $\Delta=0$, then the strong $F$-regularity (respectively, the $F$-purity) of the pair $(\Spec R, 0)$ is nothing but the strong $F$-regularity (respectively, the $F$-purity) of $R$ defined in Definition \ref{strongly F-regular rings}.
In general, the following implications hold for $F$-singularities of pairs:
\[
\xymatrix{
\textup{strongly $F$-regular} \ar@{=>}[r] & \textup{purely $F$-regular} \ar@{=>}[r] & \textup{sharply $F$-pure} \ar@{=>}[r]  & \textup{$F$-pure.}
}
\]

\begin{rem}\label{two F-purity}
Why do we consider two kinds of ``$F$-purity," sharp $F$-purity and $F$-purity? It is because both have advantages and disadvantages. 

Let $R$ be a (normal) $F$-pure ring and $\Delta$ be an effective Cartier divisor on $X=\Spec R$. 
Then 
\[
\sup\{t \ge 0 \mid (X, t \Delta)\textup{ is $F$-pure}\}=\sup\{t \ge 0 \mid (X, t \Delta)\textup{ is sharply $F$-pure}\}.
\]
This critical value is called the \textit{$F$-pure threshold} of $\Delta$ and denoted by $\mathrm{fpt}(\Delta)$. 
Since ``$F$-pure" pairs can be viewed as a positive characteristic analogue of lc pairs  (see Conjecture \ref{full correspondence}), it is expected that $(X, t_0 \Delta)$ is ``$F$-pure" where $t_0=\mathrm{fpt}(\Delta)$. 
By \cite{He2}, $(X, t_0 \Delta)$ is $F$-pure but not necessarily sharply $F$-pure. 
For example, in Example \ref{pair example} (1), $\mathrm{fpt}(\Delta)=1/3$ but $(X, (1/3)\Delta)$ is not sharply $F$-pure. 

On the other hand, sharp $F$-purity fits better into the theory of $F$-pure centers which we will discuss in \S\ref{F-adjunction section}. 
If we introduced the theory of $F$-pure centers using $F$-purity not sharp $F$-purity, then a  pathological phenomenon could happen. 
Setting $o=(x, y) \in \Spec R=X$ in Example \ref{pair example} (1), we easily see that the $F$-pure pair $(X, (1/3)\Delta)$ is strongly $F$-regular except at $o$. 
Therefore, $o$ should be the unique $F$-pure center of $(X, (1/3)\Delta)$. 
However, if we changed the definition of $F$-pure centers (Definition \ref{Fpure center def}) by replacing sharp $F$-purity with $F$-purity, 
then every irreducible component of $\Delta$ would be an $F$-pure center of $(X, (1/3)\Delta)$. 

Sharp $F$-purity has one more advantage over $F$-purity. 
If $(X, \Delta)$ is lc, then the multiplier ideal $\mathcal{J}(X, \Delta)$ associated to $(X, \Delta)$ becomes a radical ideal. 
Since test ideals are a positive characteristic analogue of multiplier ideals (see \S\ref{test ideals} for the details), it is expected that the test ideal $\tau(X, \Delta)$ associated to $(X, \Delta)$ is a radical ideal if $(X, \Delta)$ is ``$F$-pure." 
Indeed, if $(X, \Delta)$ is sharply $F$-pure, then $\tau(X, \Delta)$ becomes a radical ideal (see \cite{Sch2}). 
If $(X, \Delta)$ is only $F$-pure, then $\tau(X, \Delta)$ is not necessarily even an integrally closed ideal (see Example \ref{test ideal example} (4)). 
\end{rem}

\begin{prop}[\cite{HW}]
If $(X, \Delta)$ is strongly $F$-regular $($respectively, $F$-pure$)$, then $\lfloor \Delta \rfloor=0$ $($respectively, $\lceil \Delta \rceil$ is a reduced divisor$)$. 
\end{prop}

A criterion analogous to Proposition \ref{Fedder} holds in the pair setting. 
\begin{prop}[cf.~\textup{\cite{HW}}]\label{Fedder2}
Let $(R, \m)$ be an $F$-finite regular local ring, $f \in R$ be a nonzero element and $t>0$ be a rational number. 
Put $X=\Spec R$ and $\Delta=t \; \Div(f)$. 
\begin{enumerate}
\item $(X, \Delta)$ is $F$-pure if and only if $f^{\lfloor t (p^e-1) \rfloor} \notin \m^{[p^e]}$ for all $e \in \N$. 
\item $(X, \Delta)$ is sharply $F$-pure if and only if there exists some $e \in \N$ such that $f^{\lceil t (p^e-1) \rceil} \notin \m^{[p^e]}$. 
\item $(X, \Delta)$ is strongly $F$-regular if and only if there exists some $e \in \N$ such that $f^{\lceil t p^e \rceil} \notin \m^{[p^e]}$. 
\item 
Suppose that $t=1$ and $R/(f)$ is reduced.  
Choose an element $c \in R \setminus (f)$ such that the localization $(R/f)_c$ is strongly $F$-regular. 
Then $(X, \Delta)$ is purely $F$-regular if and only if there exists $e \in \N$ such that $cf^{p^e-1} \notin \m^{[p^e]}$. 
\end{enumerate}
\end{prop}

The following example easily follows from the above proposition. 
Note that strong $F$-regularity (respectively, $F$-purity, weak $F$-purity, pure $F$-regularity) of pairs can be checked locally (cf.~Remark \ref{F-regular remark} (2)).
\begin{ex}\label{pair example}
(1) Let $k$ be a perfect field of characteristic $3$ and $R=k[x, y]$ be the two-dimensional polynomial ring over $k$. 
Put  $f=xy(x+y)(x-y) \in R$ and $\Delta=\Div(f)$. 
Then the pair $(\Spec R, (1/3) \Delta)$ is $F$-pure but not sharply $F$-pure. 
Also, $(\Spec R, t \Delta)$ is strongly $F$-regular for every $1/3>t>0$.

(2) Let $k$ be a perfect field of characteristic $p>0$ and $R=k[x, y, z]$ be the three-dimensional polynomial ring over $k$. 
Put $f=x^3-yz(y+z) \in R$ and $\Delta=\Div(f)$. 
Then $(\Spec R,\Delta)$ is $F$-pure (equivalently, sharply $F$-pure) if and only if $p \equiv 1 \; \mathrm{mod} \; 3$. Compare this with Example \ref{F-pure example} (1). 

(3) Let $k$ be a perfect field of characteristic $p>5$ and $R=k[x, y, z]$ be the three-dimensional polynomial ring over $k$. 
Put $f=x^2+y^3+z^5 \in R$ and $\Delta=\Div(f)$. 
Then $(\Spec R,\Delta)$ is purely $F$-regular. Compare this with Example \ref{F-pure example} (2). 
\end{ex}

Let $X$ be a normal variety over a field of characteristic zero and $\Delta$ be an effective $\Q$-divisor on $X$ such that $K_X+\Delta$ is $\Q$-Cartier. 
Let $(X_A, \Delta_A)$ be a model of $(X, \Delta)$ over a finitely generated $\Z$-subalgebra $A$ of $k$. 
We say that $(X, \Delta)$ is of \textit{strongly $F$-regular type} (respectively, \textit{purely $F$-regular type}) if there exists a dense open subset $T \subset \Spec A$ such that $(X_{\mu}, \Delta_{\mu})$ is strongly $F$-regular (respectively, purely $F$-regular) for all closed points $\mu \in T$. 
We also say that $(X, \Delta)$ is of \textit{dense $F$-pure type} if there exists a dense subset of closed points $S \subset \Spec A$ such that $(X_{\mu}, \Delta_{\mu})$ is $F$-pure for all $\mu \in S$. 
If $(X, \Delta)$ is of dense $F$-pure type, then after possibly shrinking $S$, we may assume that for every $\mu \in S$, there exists some $e(\mu) \in \N$ such that $(p(\mu)^{e(\mu)}-1)\Delta_{\mu}$ is an integral divisor. 
Thus, the definition of being of dense $F$-pure type is equivalent to saying that there exists a dense subset of closed points $S \subset \Spec A$ such that $(X_{\mu}, \Delta_{\mu})$ is sharply $F$-pure for all $\mu \in S$.

As we have promised at the end of \S\ref{tight closure}, we state a correspondence between strongly $F$-regular pairs and klt pairs. 
Lemma \ref{surjectivity} plays a key role in the proof. 
\begin{thm}[\cite{Ta1}, \cite{Ta2}]\label{F-regular vs klt}
The pair $(X, \Delta)$ is klt $($respectively, plt$)$ if and only if it is of strongly $F$-regular type $($respectively, purely $F$-regular type$)$. 
\end{thm}

\begin{thm}[\cite{HW}]\label{lc pairs}
If the pair $(X, \Delta)$ is of dense $F$-pure type, then it is lc. 
\end{thm}

The converse of Theorem \ref{lc pairs} is also expected to be true, but it is wide open except in a few special cases. 

\begin{conj}\label{full correspondence}
The pair $(X, \Delta)$ is lc if and only it is of dense $F$-pure type. 
\end{conj}

Conjecture \ref{full correspondence} is known to be true in the following cases: 
\begin{enumerate}
\item (\cite{MeS1}, \cite{Ha1}, \cite{HW}) $\dim X=2$ and $\Delta$ is an integral divisor.

\item (\cite{He}) $X=\C^n, \Delta=t \;\Div(f)$ with $t \in \Q_{\ge 0}$ and $f=\sum_{\underline{i}} c_{\underline{i}}\underline{x}^{\underline{i}} \in \C[x_1, \dots, x_n]$ where the $c_{\underline{i}} \in \C \setminus \{0\}$ are algebraically independent over $\Q$. 

\item (Theorem\ref{A_3}) $\dim X=3$, $X$ has only isolated non-log-terminal points and $\Delta=0$. 
\end{enumerate}
We will explain how to show (3) briefly. 
Musta\c{t}\u{a}-Srinivas introduced in \cite{MS} the following conjecture, the so-called \textit{weak ordinarity conjecture}. 
\begin{conj}[Weak ordinarity conjecture \cite{MS}]\label{MS conj}
Let $V$ be an $n$-dimensional smooth projective variety over an algebraically closed field of characteristic zero. 
Given a model $V_A$ of $V$ over a finitely generated $\Z$-subalgebra $A$ of $k$, 
there exists a dense subset of closed points $S \subset \Spec A$ such that the natural Frobenius action on $H^n(V_{\mu}, \sO_{V_{\mu}})$ is bijective for all $\mu \in S$. 
\end{conj}

Using techniques from \cite{MS}, we can show the following. 
\begin{prop}[{\cite{BST}, \cite{Ta3}}]\label{MS implies full}
If Conjecture $\ref{MS conj}$ holds, then Conjectures $\ref{DB conj}$ and $\ref{full correspondence}$ hold as well.\footnote{More generally, it follows from a combination of Proposition \ref{MS implies full} and \cite{MiSc} that if Conjecture \ref{MS conj} holds, then being of dense $F$-pure type is equivalent to being slc pairs.} 
Conversely, if Conjecture $\ref{DB conj}$ holds, then so does Conjecture $\ref{MS conj}$. 
\end{prop}

However, since Conjecture \ref{MS conj} is open even when $\dim V=1$, it is virtually impossible to make progress on Conjecture \ref{full correspondence} by making use of Conjecture \ref{MS conj}. 
Therefore, we propose Conjectures $\mathrm{A}_n$ and $\mathrm{B}_n$, a weaker form of Conjectures \ref{full correspondence} and \ref{MS conj}, respectively.  

\begin{conjA}\label{isolated correspondence}
Let $x \in X$ be an $n$-dimensional normal $\Q$-Gorenstein singularity defined over an algebraically closed field of characteristic zero, such that $x$ is an isolated non-log-terminal point of $X$. 
Then $x \in X$ is log canonical if and only if it is of dense $F$-pure type. 
\end{conjA}

\begin{conjB}\label{FT conj}
Let $V$ be an $n$-dimensional projective variety over an algebraically closed field $k$ of characteristic zero with only rational singularities, such that $K_V$ is linearly trivial. 
Given a model $V_A$ of $V$ over a finitely generated $\Z$-subalgebra $A$ of $k$, there exists a dense subset of closed points $S \subset \Spec A$ such that the natural Frobenius action on $H^n(V_{\mu}, \sO_{V_{\mu}})$ is bijective for all $\mu \in S$. 
\end{conjB}

We can easily check that Conjecture \ref{MS conj} implies Conjecture $\mathrm{B}_n$ as follows: 
Take a resolution of singularities $\pi:\widetilde{V} \to V$. 
Since $V$ has only rational singularities, $\pi$ induces an isomorphism $H^n(V, \sO_V) \cong H^n(\widetilde{V}, \sO_{\widetilde{V}})$. 
Suppose we are given a model $\pi_A: \widetilde{V}_A \to V_A$ of $\pi$ over a finitely generated $\Z$-subalgebra $A$ of $k$. 
Applying Conjecture \ref{MS conj} to $\widetilde{V}$, we see that there exists a dense subset $S$ of closed points of $\Spec A$ such that the Frobenius action on $H^n(\widetilde{V}_{\mu}, \sO_{\widetilde{V}_{\mu}})$ is bijective for all $\mu \in S$. 
The isomorphism $H^n(\widetilde{V}_{\mu}, \sO_{\widetilde{V}_{\mu}}) \cong H^n(V_{\mu}, \sO_{V_{\mu}})$ commutes with the Frobenius action, and so we obtain the assertion. 

The following lemma is a consequence of deep arithmetic results \cite{Og}, \cite{JR}, \cite{BZ}. 
\begin{lem}\label{lemma on B_n}
$($\cite{Og}, \cite{JR}, \cite{BZ}$)$ Conjecture $\mathrm{B}_n$ holds true if $n \le 2$.  
\end{lem}

Making use of recent progress on the minimal model program (see for example \cite{BCHM}), we can prove the following statement. 
\begin{thm}[\cite{FT}]\label{A_3}
Conjecture ${\rm A}_{n+1}$ is equivalent to Conjecture ${\rm B}_{n}$.
In particular, Conjecture ${\rm A}_3$ holds true by Lemma $\ref{lemma on B_n}$. 
\end{thm}

\subsection{$F$-pure centers and $F$-adjunction}\label{F-adjunction section}
Next, we will explain the theory of $F$-adjunction, a positive characteristic analogue of Kawamata's subadjunction \cite{Ka}, introduced by Karl Schwede \cite{Sch3}, \cite{Sch4}. 

In this subsection, let $X$ be a normal integral scheme essentially of finite type over a perfect field of characteristic $p>0$ and $\Delta$ be an effective $\Q$-divisor on $X$. 
We assume in addition that $K_X+\Delta$ is $\Q$-Cartier with index\footnote{The index of $K_X+\Delta$ is the smallest positive integer $r$ such that $r(K_X+\Delta)$ is a Cartier divisor.} not divisible by $p$. 
Then there exist infinitely many $e \in \N$ such that $(p^e-1)(K_X+\Delta)$ is a Cartier divisor. 
We fix such an $e_0 \in \N$ and put $\mathcal{L}_{\Delta}=\sO_X((1-p^{e_0})(K_X+\Delta))$. 
The composite $\sO_X \to F^{e_0}_*\sO_X((p^{e_0}-1)\Delta)$ of the $e_0$-times iterated Frobenius $\sO_X \to F^{e_0}_*\sO_X$ and a natural inclusion $F^{e_0}_*\sO_X \to F^{e_0}_*\sO_X((p^{e_0}-1)\Delta)$ induces an $\sO_X$-linear map 
\[
\phi_{\Delta}:F^{e_0}_*\mathcal{L}_{\Delta} \cong \mathcal{H}\mathrm{om}_{\sO_X}(F^{e_0}_*\sO_X((p^{e_0}-1)\Delta), \sO_X) \to \mathcal{H}\mathrm{om}_{\sO_X}(\sO_X, \sO_X)=\sO_X,
\]
where the first isomorphism follows from Grothendieck-Serre duality. 
In this article, we refer $\phi_{\Delta}:F^{e_0}_*\mathcal{L}_{\Delta} \to \sO_X$ as the map corresponding to $\Delta$. 

First, we introduce the notion of $F$-pure centers, a positive characteristic analogue of log canonical  centers.  
\begin{defn}[\cite{Sch3}]\label{Fpure center def}
Let $W$ be an irreducible closed subscheme of $X$ and $\mathcal{I}_W \subset \sO_X$ denote the defining ideal sheaf of $W$. 
When $X$ is affine, we say that $W$ is a \textit{non-$F$-regular center} of $(X, \Delta)$ if for all $c \in \mathcal{I}_W$ and all $\varepsilon>0$, the pair $(X, \Delta+\varepsilon \;\mathrm{div}(c))$ is not sharply $F$-pure at the generic point $\xi_W$ of $W$. 
This definition can be easily extended to non-affine schemes by considering the same condition on each affine chart. 
We say that $W$ is an \textit{$F$-pure center} of $(X, \Delta)$ if $W$ is a non-$F$-regular center of $(X, \Delta)$ that is sharply $F$-pure at $\xi_W$. 
A minimal element (with respect to inclusion) of the set of $F$-pure centers is called a \textit{minimal $F$-pure center}. 
\end{defn}

The $F$-pure centers of $(X, \Delta)$ can be characterized in terms of $\phi_{\Delta}$. 
\begin{lem}[\cite{Sch3}]\label{center characterization}
Let $W$ be an irreducible closed subscheme of $X$ and $\mathcal{I}_W \subset \sO_X$ denote the defining ideal sheaf of $W$. 
Then $W$ is a non-$F$-regular center of $(X, \Delta)$ if and only if $\phi_{\Delta}(F^{e_0}_*(\mathcal{I}_W \mathcal{L}_{\Delta})) \subseteq \mathcal{I}_W$. 
This definition is independent of the choice of ${e_0}$. 
If $X=\Spec R$ where $R$ is an $F$-finite normal local ring, then the above condition is equivalent to saying that  
$(\phi \circ F^{e_0}_*i)(F^{e_0}_*\mathcal{I}_W) \subseteq \mathcal{I}_W$ for all $\phi \in \mathrm{Hom}_{\sO_X}(F^{e_0}_*\sO_X((p^{e_0}-1)\Delta), \sO_X)$,  
where $i: \sO_X \to \sO_X((p^{e_0}-1)\Delta)$ is a natural inclusion. 
\end{lem}

The following theorem should be compared with fact that there are only finitely many log canonical  centers. 
\begin{thm}[\cite{Sch4}]
If $(X, \Delta)$ is sharply $F$-pure, then there are at most finitely many $F$-pure centers of $(X, \Delta)$. 
\end{thm}

In order to state our $F$-adjunction formula,  we need to introduce some notation.   
Let $W$ be an $F$-pure center of $(X, \Delta)$ and we assume that $W$ is normal for now. 
Since $W$ is an $F$-pure center of $(X, \Delta)$, the map $\phi_\Delta: F^e_*\mathcal{L}_{\Delta} \to \sO_X$ induces an $\sO_W$-linear map $\phi_{\Delta_W}: F^{e_0}_*\mathcal{L}_{\Delta_W} \to \sO_W$ where $\mathcal{L}_{\Delta_W}=\mathcal{L}_{\Delta}/\mathcal{I}_{W}\mathcal{L}_{\Delta}$:  
\[
\xymatrix{
F^{e_0}_*\mathcal{L}_{\Delta} \ar@{->>}[r] \ar[d]^{\phi_{\Delta}} & F^{e_0}_*\left(\mathcal{L}_{\Delta}/\mathcal{I}_{W}\mathcal{L}_{\Delta} \right) \ar[d] \ar@{=}[r] & F^{e_0}_*\mathcal{L}_{\Delta_W} \ar[d]^{\phi_{\Delta_W}} \\
\sO_X \ar@{->>}[r]  & \sO_X/\mathcal{I}_W \ar@{=}[r] & \sO_W.
}
\]
Note that $\mathcal{L}_{\Delta_W}$ is an invertible sheaf on $W$ and $\mathcal{L}_{\Delta_W}=\sO_W((1-p^{e_0})(K_X+\Delta)|_W)$. 
Since $\phi_{\Delta_W}$ is a global section of 
\[
\mathcal{H}\mathrm{om}_{\sO_W}(F^{e_0}_*\mathcal{L}_{\Delta_W}, \sO_W) \cong F^{e_0}_*\sO_W((p^{e_0}-1)(K_X+\Delta)|_W-(p^{e_0}-1)K_W), 
\]
where the isomorphism follows from Grothendieck-Serre duality, we have a corresponding effective divisor $\Gamma_W$ on $W$ such that  
\[
\Gamma_W \sim (p^{e_0}-1)(K_X+\Delta)|_W-(p^{e_0}-1)K_W.
\]
Then $\Delta_W:=(1/(p^{e_0}-1))\Gamma_W$ is an effective $\Q$-divisor on $W$ satisfying that 
\[
K_W+\Delta_W \sim_{\Q} (K_X+\Delta)|_W.
\]
\begin{thm}[\cite{Sch4}]
\label{F-subadjunction}
In the above notation, the following holds: 
\begin{enumerate}
\item 
$(W, \Delta_W)$ if sharply $F$-pure if and only if so is $(X, \Delta)$. 
\item 
$(W, \Delta_W)$ is strongly $F$-regular if and only if $W$ is a minimal $F$-pure center of $(X, \Delta)$. 
\item 
There is a natural bijection between the $F$-pure centers of $(W, \Delta_W)$ and the $F$-pure centers of $(X, \Delta)$ properly contained in $W$ $($as topological spaces$)$. 
\end{enumerate}
\end{thm}

\begin{ex}
Let $X=\mathbb{A}^3_k=\Spec k[x, y, z]$ and $\Delta=\Div(x^2z-y^2)$, where $k$ is a perfect field of characteristic $p \ge 3$.  
Since $(x^2z-y^2)^{p-1} \notin (x^p, y^p, z^p)$, by  Proposition \ref{Fedder2} (2), the pair $(X, \Delta)$ is sharply $F$-pure. 
We will show that $W=V((x, y))$ is a minimal $F$-pure center of $(X, \Delta)$. 

We may assume that $e_0=1$.
Then $\mathcal{L}_{\Delta} \cong \sO_X$ and the map $\phi_{\Delta}: F_*\sO_X \to \sO_X$ corresponding to $\Delta$ is nothing but the composite map 
\[
F_*\sO_X \xrightarrow{\times F_*(x^2z-y^2)^{p-1}} F_*\sO_X \xrightarrow{\mathrm{Tr}_F} \sO_X,  
\]
where $\mathrm{Tr}_F: F_*\sO_X \to \sO_X$ is the $\sO_X$-linear map sending $F_*(xyz)^{p-1}$ to $1$ and all other lower-degree monomials to zero. 
Since 
\[\phi_{\Delta}(F_*(x, y))=\mathrm{Tr}_F(F_*((x, y)(x^2z-y^z)^{p-1})) \subset \mathrm{Tr}_F(F_*(x^p, y^p)) \subset (x, y),\]
$W=V((x, y))$ is an $F$-pure center of $(X, \Delta)$ by Lemma \ref{center characterization}. 
The map $\phi_{\Delta}$ induces an $\sO_W$-linear map $\phi_{\Delta_W}:F_*\sO_W \to \sO_W$ sending $F_*z^{(p-1)/2}$ to 1, from which it follows that $\Delta_W=(1/2)\Div(z)$. 
The pair $(W, \Delta_W)=(\Spec k[z], (1/2) \Div(z))$ is strongly $F$-regular by Lemma \ref{Fedder2} (3),  and so we conclude from Theorem \ref{F-subadjunction} that $W$ is a minimal $F$-pure center of $(X, \Delta)$. 
\end{ex}

When $W$ is not normal, we take the normalization $W^{\rm N}$ of $W$. 
Since $\phi_{\Delta_W}$ extends to a unique $\sO_{W^{\rm N}}$-linear map $\phi_{\Delta_{W^{\rm N}}}: F^e_*\sO_{W^{\rm N}} \to \sO_{W^{\rm N}}$, 
we can define an effective $\Q$-divisor $\Delta_{W^{\rm N}}$ on $W^{\rm N}$ satisfying that $K_{W^{\rm N}}+\Delta_{W^{\rm N}} \sim_{\Q} (K_X+\Delta)|_{W^{\rm N}}$ similarly. 
The pair $(X, \Delta)$ is, however, not necessarily sharply $F$-pure if $(W^{\rm N}, \Delta_{W^{\rm N}})$ is sharply $F$-pure. 

\begin{ex}[\cite{Sch4}]\label{normalization}
Let $k$ be a perfect field of characteristic two. 
Put $X=\mathbb{A}^3_k=\Spec k[x, y, z]$ and $\Delta=\Div(x^2z+y^2)$. 
Then $W=\Delta$ is an $F$-pure center of $(X, \Delta)$. 
Since $\sO_W=k[x, y, z]/(x^2z+y^2) \cong k[u, uv, v^2] \hookrightarrow k[u, v]$, 
the normalization $W^{\rm N}$ of $W$ can be identified with $\Spec k[u, v]$. 
We may assume that $e_0=1$, and then the map $\phi_{\Delta}: F_*\sO_X \to \sO_X$ corresponding to $\Delta$ is the composite map 
\[
F_*\sO_X \xrightarrow{\times F_*(x^2z+y^2)} F_*\sO_X \xrightarrow{\mathrm{Tr}_F} \sO_X \quad F_*(xy) \mapsto F_*(x^3yz+xy^3) \mapsto x,
\]
where $\mathrm{Tr}_F: F_*\sO_X \to \sO_X$ is the map sending $F_*(xyz)$ to $1$ and all other lower-degree monomials to zero. 
It induces an $\sO_{W^{\rm N}}$-linear map $\phi_{\Delta_{W^{\rm N}}}: F^e_*\sO_{W^{\rm N}} \to \sO_{W^{\rm N}}$ sending $F_*v$ to $1$, which implies that  $\Delta_{W^{\rm N}}=\Div(u)$. 
It is easily checked from Proposition \ref{Fedder2} (2) that $(W^{\rm N}, \Delta^{\rm N})=(\Spec k[u, v], \Div(u))$ is sharply $F$-pure but $(X, \Delta)=(\Spec k[x, y, z], \Div(x^2z+y^2))$ is not.
\end{ex}

Such a pathology can be avoided by assuming that $W$ has hereditary surjective trace.  
Let $R$ be an $F$-finite reduced local ring, $R^{\rm N}$ be the normalization of $R$ and $\mathfrak{c}$ be its conductor ideal.  
Note that $\mathfrak{c}$ is an ideal of both rings $R$ and $R^{\rm N}$. 
We say that $R$ has \textit{hereditary surjective trace} if there exist minimal associated prime ideals of $\mathfrak{c}$, $\mathfrak{p} \subset R$ and $\mathfrak{q} \subset R^{\rm N}$, satisfying the following three conditions: (i) $R \cap \mathfrak{q}=\mathfrak{p}$, (ii) the induced trace map $\mathrm{Tr}: (R^{\rm N}/\mathfrak{q})^{\rm N} \to (R/\mathfrak{p})^{\rm N}$ is surjective, and (iii) $R/\mathfrak{p}$ also has hereditary surjective trace.  
Since $\dim R/\mathfrak{p} <\dim R$, this definition is well-defined.  
We say that an $F$-finite reduced scheme $X$ has hereditary surjective trace if the local ring $\sO_{X, x}$ has hereditary surjective trace for every $x \in X$. 
Although the definition looks complicated at first glance, it is known by \cite{MiSc} that any reduced scheme of finite type over an algebraically closed field of characteristic zero has hereditary surjective trace after reduction to characteristic $p \gg 0$. 
\begin{thm}[\textup{\cite{MiSc}}]
Let the notation be as above and $W$ be a $($not necessarily normal$)$ $F$-pure center of $(X, \Delta)$. 
If $W$ has hereditary surjective trace, then $(W^{\rm N}, \Delta_{W^{\rm N}})$ is sharply $F$-pure if and only if so is $(X, \Delta)$. 
\end{thm}


\section{Test ideals}\label{test ideals}

It is usually difficult to compute the tight closure of a given ideal directly from the definition. 
Therefore, the notion of test elements, elements for testing membership in tight closure, was introduced.
The classical test ideal $\tau(R)$ is defined to be the ideal generated by all test elements (see Remark \ref{test remark}), which means that test ideals originally come from the theory of tight closure. 
However, once it turned out that they were a positive characteristic analogue of multiplier ideals, test ideals quickly began finding applications in their own right. 
In this section, after explaining the definition and basic results of test ideals, we will mention two of their applications. One is to commutative algebra and the other is to algebraic geometry.  

Strictly speaking, there are two kinds of test ideals, finitistic test ideals and big test ideals.\footnote{It is conjectured that they coincide with each other. This conjecture is known to be true if the ring is normal and $\Q$-Gorenstein (see \cite{HY}). }
Finitistic test ideals have been considered as more natural from the point of view of tight closure theory, and they are often referred to simply as ``test ideals" in the literature. 
However, since big test ideals are easier to treat and have more applications, we focus on big test ideals in this article. 
Although the big test ideal of an ideal $\ba$ is often denoted by $\widetilde{\tau}(\ba)$ or $\tau_b(\ba)$ in the literature, we denote it simply by $\tau(\ba)$ and refer to it simply as the ``test ideal" of $\ba$.

\subsection{Definition and basic properties}\label{test def basic}
While we have considered a pair consisting of a normal variety and a $\Q$-divisor on it in the previous section,  we consider a pair consisting of a reduced ring and its ideal in this section. 
Let $R$ be an $F$-finite reduced ring, $\ba \subset R$ be an ideal satisfying that $\ba \cap R^{\circ} \ne \emptyset$ and $t>0$ be a real number. 

Schwede \cite{Sch3} gave a characterization of test ideals. 
We take this as the definition of test ideals. 
\begin{defn}[\cite{Sch3}]\label{test ideal def}
The \textit{test ideal} $\tau(\ba^t)$ of $\ba$ with exponent $t$ is defined to be the unique smallest ideal $J$ of $R$ that satisfies  the following two conditions: 
\begin{enumerate}
\item[(i)]  $\phi^{(e)}(F^e_*(J\ba^{\lceil t(p^e-1) \rceil})) \subseteq J$ for all $e \in \N$ and all $\phi^{(e)} \in \Hom_R(F^e_*R, R)$, 
\item[(ii)]  $J \cap R^{\circ} \ne \emptyset$.
\end{enumerate}
When $\ba=R$, we denote this ideal by $\tau(R)$. 
\end{defn}

The existence of the test ideal $\tau(\ba^t)$ is not clear from definition, and we use the notion of test element to describe $\tau(\ba^t)$ more explicitly. 

\begin{defn}[\cite{Ho1}.~cf.~\cite{HH0}]\label{test element}
We say that $c \in R^{\circ}$ is a \textit{test element} for $R$ if for every $d \in R^{\circ}$, there exist some $e \in \N$ and $\phi^{(e)} \in \Hom_R(F^e_*R, R)$ such that $c=\phi^{(e)}(F^e_*d)$. 
\end{defn}

The following lemma is useful for finding a test element. 
\begin{lem}[\cite{Ho1}, \cite{Ta4}.~cf.~\cite{HH0}]\label{test vs jacobi}
\begin{enumerate}
\item 
Let $c \in R^{\circ}$ be an element such that the localization $R_c$ with respect to $c$ is strongly $F$-regular. 
Then some power $c^n$ of $c$ is a test element for $R$. 

\item Suppose that $R$ is essentially of finite type over an $F$-finite field $k$, and denote by $\mathfrak{J}(R/k)$ the Jacobian ideal of $R$ over $k$. 
Then every element of $\mathfrak{J}(R/k) \cap R^{\circ}$ is a test element for $R$. 
\end{enumerate}
\end{lem}

Now we give an explicit description of test ideals using of test elements. 
\begin{lem}[\cite{HT}]\label{lemma HT}
Let $c \in R^{\circ}$ be a test element for $R$. Then 
\[
\tau(\ba^t)=\sum_{e \ge 0}\sum_{\phi^{(e)}} \phi^{(e)}(F^e_*(c\ba^{\lceil tp^e \rceil})), 
\]
where $\phi^{(e)}$ runs through all elements of $\Hom_R(F^e_*R, R)$. 
\end{lem}

\begin{rem}[\cite{HH0}]\label{test remark}
The test ideal $\tau(R)$ coincides with the ideal generated by all the test elements for $R$. 
The name of ``test ideal" comes from this fact. 
$\tau(R)$ can be also defined in terms of tight closure. 
For simplicity, we assume that $(R, \m)$ is an $F$-finite reduced local ring, and let $E=E_R(R/\m)$ be the injective hull of the residue field $R/\m$ of $R$. 
If we denote by $0_E^*$ the tight closure of the zero submodule in $E$ (see \S\ref{tight closure} for its definition), then $\tau(R)=\mathrm{Ann}_R \ 0_E^*$. 
\end{rem}

We list some properties of test ideals that follow immediately from Definition \ref{test ideal def} and Lemma \ref{lemma HT}. 
\begin{prop}\label{test ideal basic}
Let $\bb$ be an ideal of $R$ such that $\bb \cap R^{\circ} \ne \emptyset$ and $s>0$ be a real number. 
\begin{enumerate}
\item $\tau(R)\ba \subseteq \tau(\ba)$. 
\item If $\ba \subseteq \bb$, then $\tau(\ba^t) \subseteq \tau(\bb^t)$. 
If we assume in addition that $\bb$ is contained in the integral closure $\overline{\ba}$ of $\ba$, then $\tau(\ba^t)=\tau(\bb^t)$. 
\item If $s<t$, then $\tau(\ba^s) \supseteq \tau(\ba^t)$. 
Also, $\tau((\ba^m)^t)=\tau(\ba^{mt})$ for every $\m \in \N$. 
\item 
There exists some $\epsilon>0$, depending on $t$, such that $\tau(\ba^s)=\tau(\ba^t)$ for all $s \in [t, t+\epsilon]$. 
\item 
$\tau(R)=R$ if and only if $R$ is strongly $F$-regular.
\item 
If $W \subset R$ is a multiplicative set, then $\tau(\ba^t)R_W=\tau((\ba R_W)^t)$. 
\item 
If $(R, \m)$ is a local ring, then $\tau((\ba \widehat{R})^t)=\tau(\ba^t) \widehat{R}$, where $\widehat{R}$ is the $\m$-adic completion of $R$. 
\end{enumerate}
\end{prop}

\begin{ex}\label{test ideal example}
(1) Suppose that $R$ is a domain essentially of finite type over an $F$-finite field $k$, and denote by $\mathfrak{J}(R/k)$ the Jacobian ideal of $R$ over $k$. 
It then follows from Lemma \ref{test vs jacobi} and Remark \ref{test remark} that $\mathfrak{J}(R/k) \subseteq \tau(R)$. 

(2) Let $(R, \m)$ be an $F$-finite $F$-pure local ring of characteristic $p>0$. Suppose that the local ring $R_{\fkp}$ is strongly $F$-regular for all prime ideals $\fkp \ne \m$, but $R$ is not. Then $\tau(R)=\m$. 
Indeed, by the $F$-purity of $R$, there is an $R$-module homomorphism $\varphi:F_*R \to R$ sending $F_*1$ to $1$. 
For every $x \in \m$, some power $x^n$ of $x$ is a test element for $R$ by Lemma \ref{test vs jacobi} (1). Take a sufficiently large $e \in \N$ such that $p^e \ge n$, and consider the following $R$-linear map:  
\[
\phi:F^e_*R \xrightarrow{\times F^e_*x^{p^e-n}} F^e_*R \xrightarrow{\varphi^{e}} R \quad F^e_*x^n \mapsto F^e_*x^{p^e} \mapsto x.
\]
It then follows from Lemma \ref{lemma HT} that $x=\phi(F^e_*x^n) \in \tau(R)$. 
Thus, $\m \subset \tau(R)$. 
Since $R$ is not strongly $F$-regular, we conclude from Proposition \ref{test ideal basic} (5) that $\tau(R)=\m$. 

Suppose that $(R, \m)=k[[X,Y,Z]]/(X^3-YZ(Y+Z))$ where $k$ is a perfect field of characteristic $p>0$. 
If $p \equiv 1 \; \mathrm{mod}\; 3$, then $R$ is $F$-pure by Example \ref{F-pure example} and then $\tau(R)=\m$ by the above argument.\footnote{In fact, even when $p \equiv 2 \; \mathrm{mod}\; 3$, one has $\tau(R)=\m$.}

(3) (\cite{HY}) Let $R=k[x_1, \dots, x_d]$ be a polynomial ring over an $F$-finite field $k$ of characteristic $p>0$ and $\ba$ be a monomial ideal of $R$. Then 
\[\tau(\ba^t)=\langle x^v \; | \; v+(1,\dots, 1) \in \mathrm{Int}(P(\ba)) \rangle,\]
where $P(\ba) \subset \R^d$ is the Newton polytope of $\ba$.\footnote{The Newton polytope $P(\ba)$ of $\ba$ is the convex hull of the set of exponent vectors of the monomial generators of $\ba$ in $\R^d$.}

(4) (\cite{MY}) Let $R=k[x_1, \dots, x_d]$ be a polynomial ring over a perfect field $k$ of characteristic $p>0$, and fix a polynomial $f \in R$.  
If there exist some $e_0 \in \N$ and $g_{i_1, \dots, i_d} \in R$ with $0 \le i_1, \dots, i_d <p^{e_0}$ such that 
\[
f=\sum_{0 \le i_1, \dots, i_d<p^{e_0}} g_{i_1, \dots, i_d}^{p^{e_0}}x_1^{i_1} \cdots x_d^{i_d}, 
\]
then $\tau((f)^{1/p^{e_0}})=\sum_{0 \le i_1, \dots, i_d<p^{e_0}} Rg_{i_1, \dots, i_d}$. 

Suppose that $p=2$, $R=k[x, y, z]$ and $f=x^2+y^5+z^5 \in R$. 
Put $\Delta=\Div(f)$. Since $f=x^2 \cdot 1+(y^2)^2 \cdot y+(z^2)^2 \cdot z$, one has 
\[
\tau(\Spec R, (1/2) \Delta)=\tau((f)^{1/2})=(x, y^2, z^2),
\] 
which is not a radical ideal. 
On the other hand, $(\Spec R, (1/2) \Delta)$ is $F$-pure by Proposition \ref{Fedder2}. 
We refer to Remark \ref{two F-purity} for an explanation of this example. 
\end{ex}

We will state three important local properties of test ideals after we introduce the notion of test ideals  associated to several ideals. 
Let $\bb$ be an ideal of $R$ such that $\bb \cap R^{\circ} \ne \emptyset$ and $s>0$ be a real number. 
The test ideal $\tau(\ba^s \bb^t)$ is the unique smallest ideal $J$ that satisfies two conditions: (i) $\phi^{(e)}(F^e_*(J\ba^{\lceil s(p^e-1) \rceil}\bb^{\lceil t(p^e-1) \rceil})) \subseteq J$ for all $e \in \N$ and all $\phi^{(e)} \in \Hom_R(F^e_*R, R)$, and (ii) $J \cap R^{\circ} \ne \emptyset$. 
\begin{thm}[\cite{Ta4}]\label{subadditivity}
Let $k$ be a perfect field of characteristic $p>0$ and $R$ be an integral domain essentially of finite type over $k$. 
If we denote by $\mathfrak{J}(R/k)$ the Jacobian ideal of $R$ over $k$, then 
$\mathfrak{J}(R/k) \tau(\ba^s) \tau(\bb^t) \subseteq \tau(\ba^s \bb^t)$ for all nonzero ideals $\ba, \mathfrak{b} \subset R$ and all real numbers $s, t>0$. 
\end{thm}

\begin{thm}[\cite{HT}]\label{skoda}
If $\ba$ is generated by at most $l$ elements, then $\tau(\ba^l)=\tau(\ba^{l-1})\ba$. 
\end{thm}

Theorem \ref{skoda} gives an alternative proof of the theorem of Brian\c{c}on-Skoda (Theorem \ref{BS thm}) when the ring is strongly $F$-regular. 
Let $R$ be an $F$-finite strongly $F$-regular local ring and $I \subseteq R$ be an ideal generated by $n$ elements. 
It then follows from Proposition \ref{test ideal basic} (1), (2), (5) and Theorem \ref{skoda} that for every $w \in \N$,  
\[
\overline{I^{n+w-1}}= \tau(R) \overline{I^{n+w-1}} \subseteq \tau(\overline{I^{n+w-1}})= \tau(I^{n+w-1})=\tau(I^{n-1}) I^{w} \subseteq I^w.
\]

Next, we introduce the notion of $F$-jumping numbers. 
A real number $t>0$ is said to be an \textit{$F$-jumping number} of $\ba$ if $\tau(\ba^{t}) \subsetneq \tau(\ba^{t-\epsilon})$ for all $\epsilon \in (0, t)$. 
Note that the family of test ideals $\tau(\ba^t)$ of a fixed ideal $\ba$ is right continuous in $t$. 
\begin{thm}[\cite{ScTu2}]\label{jumping number}
Suppose that $R$ is an $F$-finite normal $\Q$-Gorenstein domain. 
Then the set of $F$-jumping numbers of $\ba$ is a discrete set of rational numbers. 
\end{thm}

Theorems \ref{subadditivity}, \ref{skoda} and \ref{jumping number} are evidence that test ideals have many similar properties to those of multiplier ideals. 
However, test ideals behave totally differently from multiplier ideals in some ways. 
For example, test ideals are not necessarily integrally closed as we have seen in Example \ref{test ideal example} (4), while multiplier ideals are always integrally closed. 

Lemma \ref{surjectivity} enables us to show a correspondence between test ideals and multiplier ideals. Before stating the theorem, we recall the definition of multiplier ideals. 
Let $X$ be a normal $\Q$-Gorenstein integral scheme essentially of finite type over a field of characteristic zero and $\ba$ be a nonzero ideal sheaf on $X$. 
We take a log resolution $\pi:\widetilde{X} \to X=\Spec R$ of $(X, \ba)$, that is, a proper birational morphism with $\widetilde{X}$ nonsingular such that $\ba \sO_{\widetilde{X}}=\sO_X(-F)$ is invertible and that $\mathrm{Exc}(\pi)$ and $\mathrm{Exc}(\pi) \cup \mathrm{Supp}(F)$ are simple normal crossing divisors. 
For every real number $t>0$, the \textit{multiplier ideal} $\mathcal{J}(\ba^t)$ of $\ba$ with exponent $t$ is defined to be
\[
\mathcal{J}(\ba^t)=\mathcal{J}(X, \ba^t)=\pi_*\sO_{\widetilde{X}}(\lceil K_{\widetilde{X}}-\pi^*K_X-tF \rceil) \subset \sO_X.
\]
This definition is independent of the choice of the log resolution $\pi$. 
\begin{thm}[\cite{HY}, cf.~\cite{Ha4}, \cite{Sm2}]\label{test vs mult}
Let $R$ be a normal $\Q$-Gorenstein domain essentially of finite type over a field $k$ of characteristic zero, $\ba$ be a nonzero ideal of $R$ and $t>0$ be a real number. 
Given a model $(R_A, \ba_A)$ of $(R, \ba)$ over a finitely generated $\Z$-subalgebra $A$ of $k$, 
there exists a dense open subset $S \subset \Spec A$ such that $\tau(\ba_{\mu}^t)=\mathcal{J}(\ba^t)_{\mu}$ for all closed points $\mu \in S$. 
\end{thm}

The set $S$ in Theorem \ref{test vs mult} depends on $t$ in general. 
\begin{ex}\label{cusp example}
Let $R=\C[x, y]$ be the two-dimensional polynomial ring over $\C$ and $\ba=(x^2+y^3) \subset R$ be the principal ideal generated by $x^2+y^3$.
Then $R_{\Z}=\Z[x, y]$ and $\ba_{\Z}=(x^2+y^3) \subset R_{\Z}$ is a model of $R$ and $\ba$ over $\Z$, respectively. 
Suppose that $t>0$ is a real number and $p$ is a prime number. 
It is not hard to check that 
\[
\mathcal{J}(\ba^t)_p=\tau(\ba_p^t)=R_p \textup{ if and only if }
\left\{
\begin{array}{ll}
1/2>t & (p=2)\\
2/3>t & (p=3)\\
5/6>t & (p \equiv 1  \; \mathrm{mod} \; 3)\\
5/6-1/(6p)>t & (p \equiv 2  \; \mathrm{mod} \; 3).
\end{array}
\right.
\]
Hence, if $t<5/6$, then 
\[S_t=\left\{p \in \Spec \Z \, \bigg| \, p \ge 5 \textup{ and } \frac{5}{6}-\frac{1}{6p}>t\right\} \cup \{0\}\]
is a dense open subset of $\Spec \Z$, depending on $t$, such that $\mathcal{J}(\ba^t)_p=\tau(\ba_p^t)$ for all prime numbers $p \in S_t$. 
On the other hand, it is known that 
\[
S=\{p \in \Spec \Z \mid p \equiv 1 \; \mathrm{mod} \; 3\}
\]
is not a dense open subset of $\Spec \Z$ but a dense subset of closed points of $\Spec \Z$ such that $\mathcal{J}(\ba^{\lambda})_p=\tau(\ba_p^{\lambda})$ for all $p \in S$ and all real numbers $\lambda>0$. 
\end{ex}

As Example \ref{cusp example} suggests, it is expected that one can make $S$ in Theorem \ref{test vs mult} independent of $t$ by replacing the condition ``a dense open subset" with ``a dense subset of closed points." 
\begin{conj}[\cite{MS}]\label{test ideal conj}
Let $X$ be an $n$-dimensional nonsingular variety over an algebraically closed field $k$ of characteristic zero and $\ba$ be a nonzero ideal on $X$.
Given a model $X_A$ of $X$ over a finitely generated $\Z$-subalgebra $A$ of $k$, there exists a  dense subset of closed points $S \subset \Spec A$ such that $\tau(\ba_{\mu}^{\lambda})=\mathcal{J}(\ba^{\lambda})_{\mu}$ for all $\mu \in S$ and all $\lambda >0$. 
\end{conj}
 
\begin{thm}[\cite{MS}, \cite{Mu}]
Conjecture $\ref{test ideal conj}$ is equivalent to Conjecture $\ref{MS conj}$. 
\end{thm}

Summing up, we have seen that the following implications holds for the conjectures discussed in this article. 
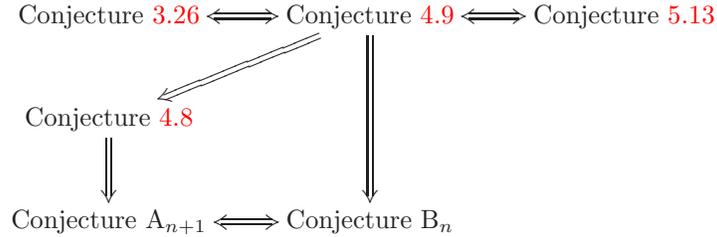
\[
\xymatrix{
\textup{Conjecture }\ref{DB conj} \ar@{<=>}[r] & \textup{Conjecture }\ref{MS conj} \ar@{<=>}[r]  \ar@{=>}[dd] \ar@{=>}[dl] & \textup{Conjecture }\ref{test ideal conj}  \\
\textup{Conjecture }\ref{full correspondence} \ar@{=>}[d] & & \\
\textup{Conjecture ${\rm A}_{n+1}$}  \ar@{<=>}[r]  & \textup{Conjecture ${\rm B}_{n}$} &
}
\]


\subsection{Asymptotic test ideals and their applications}
In this subsection, we will explain the theory of asymptotic test ideals and their applications. 

\subsubsection{Application to symbolic powers}\label{symbolic power}
Suppose that $R$ is an $F$-finite integral domain. 
We say that $\ba_{\bullet}=\{\ba_m\}_{m \in \N}$ is a \textit{graded family of ideals} in $R$ if the $\ba_m$ are ideals of $R$ such that $\ba_m \cdot \ba_n \subseteq \ba_{m+n}$ for all $m, n \in \N$. 
Let $\ba_{\bullet}=\{\ba_m\}_{m \in \N}$ be a graded family of ideals in $R$ and $t>0$ be a real number. 
It follows from Proposition \ref{test ideal basic} (3) that 
\[\tau(\ba_m^{t/m})=\tau((\ba_{m}^n)^{t/(mn)}) \subseteq \tau(\ba_{mn}^{{t}/{(mn)}})\]
for all $m, n \in \N$. 
Since $R$ is Noetherian (all rings are Noetherian throughout this article), by the above inclusion, we see that the family of test ideals $\{\tau(\ba_m^{t/m})\}_{m \in \N}$ has a unique maximal element with respect to inclusion. 
This maximal element is called the \textit{asymptotic test ideal} of $(R, \ba_{\bullet}^t)$ and denoted by $\tau(\ba_{\bullet}^t)$.

We can define an analogous notion for several graded families of ideals. 
Let $\ba_{\bullet}=\{\ba_m\}_{m \in \N}$, $\bb_{\bullet}=\{\bb_m\}_{m \in \N}$ be graded families of ideals in $R$ and $s, t>0$ be real numbers. 
The asymptotic test ideal $\tau(\ba_{\bullet}^s \bb_{\bullet}^t)$ is the unique maximal element among the family of ideals $\{\tau(\ba_m^{s/m} \bb_m^{t/m})\}_{m \in \N}$. 

We list some properties of asymptotic test ideals that easily follow from Proposition \ref{test ideal basic} and Theorems \ref{subadditivity}. 
\begin{prop}\label{asymptotic basic}
\begin{enumerate}
\item $\tau(\ba_{\bullet}^t) \ba_m \subseteq \tau(\ba_{\bullet}^{m+t})$ for every $m \in \N$.  
\item If $s<t$, then $\tau(\ba_{\bullet}^s) \supseteq \tau(\ba_{\bullet}^t)$. 
\item There exists some $\epsilon>0$, depending on $t$, such that $\tau(\ba_{\bullet}^s)=\tau(\ba_{\bullet}^{t})$ for all $s \in[t, t+\epsilon]$. 
\item Suppose that $R$ is essentially of finite type over a perfect field $k$, and denote by $\mathfrak{J}(R/k)$ the Jacobian ideal of $R$ over $k$. 
Then for all $m \in \N$, one has 
\[
\mathfrak{J}(R/k)^{m-1}\tau(\ba_{\bullet}^{mt}) \subseteq \tau(\ba_{\bullet}^t)^m.
\]
\end{enumerate}
\end{prop}

As one of the applications of asymptotic test ideals, we obtain a uniform bound for the growth of symbolic powers of ideals.  
First we recall the definition of symbolic powers of ideals. 
Let $\ba$ be an ideal of $R$, and put $W=R \setminus \bigcup_{P \in \mathrm{Ass}(R/\ba)} P$ where $P$ runs though the associated primes of $\ba$. 
Then for each $n \in \N$, the $n$-th symbolic power $\ba^{(n)}$ of $\ba$ is the contraction 
 $\ba^nR_W \cap R$, where $R_W$ denotes the localization of $R$ with respect to the multiplicative set $W$. 
 In particular, if $P$ is a prime ideal of $R$, then $P^{(n)}=P^nR_P \cap R$. 
 Note that the collection of symbolic powers $\ba_{\bullet}=\{\ba^{(m)}\}_{m \in \N}$ is a graded family of ideals in $R$. 

Since it is obvious from definition that $\ba^n \subseteq \ba^{(n)}$, it is natural to ask how large the $\ba^{(n)}$ are when compared with the $\ba^n$. 
The following theorem gives an answer to this question. 

\begin{thm}[\cite{Ta4}. cf.~\cite{ELS}, \cite{HH2}, \cite{Ha3}]\label{symbolic vs ordinary}
Let $R$ be an integral domain essentially of finite type over a perfect field $k$ of characteristic $p>0$ $($respectively, of characteristic zero$)$, and denote by $\mathfrak{J}(R/k)$ the Jacobian ideal of $R$ over $k$. 
Let $\ba$ be a nonzero ideal of $R$ and $h$ denote the largest analytic spread\,\footnote{If $(A, \m)$ is a local ring and $I \subseteq \m$ is an ideal of $A$, then the \textit{analytic spread} $\ell(I)$ of $I$ is defined to be the Krull dimension of the ring $A/\m \otimes_A (\bigoplus_{n \ge 0} I^n/I^{n+1})$. In general, $\mathrm{ht}\; I \le \ell(I) \le \dim A$.}of $\ba R_P$ as $P$ runs through the associated primes of $\ba$. 
Then for every integer $m \ge 0$ and every $n \in \N$, we have 
$\mathfrak{J}(R/k)^n \ba^{(hn+mn)} \subseteq (\ba^{(m+1)})^n$. 
In particular, $\mathfrak{J}(R/k)^n \ba^{(hn)} \subseteq \ba^n$. 
When $R$ is strongly $F$-regular $($respectively, $\Spec R$ has only log terminal singularities$)$, 
we can decrease the exponent on the Jacobian ideal $\mathfrak{J}(R/k)$ by one, in other wards, $\mathfrak{J}(R/k)^{n-1} \ba^{(hn+mn)} \subseteq (\ba^{(m+1)})^n$. 
\end{thm}

We give a sketch of the proof in the case of positive characteristic, to which the characteristic zero case can be reduced using the techniques from \S\ref{reduction}.
Assume that $\ba$ is a prime ideal $P$ of $R$ and $m=0$ for simplicity. 
We assume in addition that the residue field of the local ring $R_P$ is infinite. 
Then $h$ is nothing but the analytic spread of the maximal ideal $PR_P$ and it is known by the general theory of integral closure that there exists a proper ideal $J$ of $R_P$ generated by $h$ elements such that $\overline{J}=PR_P$. 
Let $P_{\bullet}=\{P^{(m)}\}_{m \in \N}$ be the graded family of symbolic powers of $P$. 
By the definition of asymptotic test ideals, $\tau(P_{\bullet}^{h})=\tau((P^{(l)})^{h/l})$ for sufficiently divisible $l \in \N$.  
It follows from Proposition \ref{test ideal basic} (2), (3) and (6) that 
\[\tau((P^{(l)})^{h/l})R_P=\tau((P^{(l)}R_P)^{h/l})=\tau((P^l R_P)^{h/l})=\tau((PR_P)^{h})=\tau(J^{h}).
\]
Applying Theorem \ref{skoda} to this, one has $\tau(P_{\bullet}^{h})R_P=\tau(J^h)=\tau(J^{h-1})J \subset PR_P$, which implies that $\tau(P_{\bullet}^{h}) \subset P$. 
The proof of Theorem \ref{symbolic vs ordinary} now follows from a combination of Example \ref{test ideal example} (1) and Proposition \ref{asymptotic basic} (1), (4): 
\[
\mathfrak{J}(R/k)^n P^{(hn)} \subseteq \mathfrak{J}(R/k)^{n-1}\tau(R) P^{(hn)}  \subseteq \mathfrak{J}(R/k)^{n-1} \tau(P^{hn}_{\bullet}) \subseteq \tau(P^{h}_{\bullet})^n \subseteq P^n.
 \]


\subsubsection{Application to asymptotic base loci}
As another application of asymptotic test ideals, we will explain a description of asymptotic base loci in positive characteristic due to Mircea Musta\c{t}\u{a} \cite{Mu2}. 

Let $X$ be an $F$-finite normal integral scheme and $\Delta$ be an effective $\Q$-divisor on $X$. 
Let $\ba$ be a nonzero ideal sheaf on $X$ and $t>0$ be a real number.  
As a variant of the test ideal $\tau(\ba^t)$ defined in Definition \ref{test ideal def}, we define the test ideal $\tau((X, \Delta), \ba^t)$ associated to the triple $(X, \Delta, \ba^t)$ as follows: 
If $X=\Spec R$ is an affine scheme, then $\tau((X, \Delta), \ba^t)$ is the unique smallest nonzero ideal $J$ of $\sO_X$ such that 
\[(\phi^{(e)}\circ F^e_*i)(F^e_*(J\ba^{\lceil t(p^e-1) \rceil})) \subseteq J\]
for all $e \in \N$ and all $\phi^{(e)} \in \Hom_{\sO_X}(F^{e}_*\sO_X(\lceil (p^{e}-1)\Delta \rceil), \sO_X)$, 
where $i:\sO_X \to \sO_X(\lceil (p^{e}-1)\Delta \rceil)$ is a natural inclusion. 
In the general case, $\tau((X, \Delta), \ba^t)$ is the ideal sheaf on $X$ obtained by gluing the constructions on affine charts. 
Note that the test ideal $\tau((X, \Delta), \ba^t)$ coincides with the test ideal $\tau(\ba^t)$ in Definition \ref{test ideal def} in the case where $X=\Spec R$ is affine and $\Delta=0$. 
Thus, we denote the ideal sheaf $\tau((X, \Delta), \ba^t)$ simply by $\tau(X,\ba^t)$ when $\Delta=0$.  

We define the trace map $\mathrm{Tr}_F^e:F^e_*\omega_X \to \omega_X$ of the $e$-times iterated Frobenius on $X$ as in \S\ref{trace section}. Namely, $\mathrm{Tr}_F^e$ is the $\omega_X$-dual of $F^e:\sO_X \to F^e_*\sO_X$. 
When $X$ is regular (nonsingular), there is a simple characterization of test ideals in terms of the trace maps. 

\begin{prop}[\cite{Mu2}, cf.~\cite{BMS}]\label{smooth test ideal}
Suppose that $X$ is a nonsingular projective variety over a perfect field of characteristic $p>0$. 
For an ideal sheaf $\bb$ on $X$ and an $e \in \N$, the ideal sheaf $\bb^{[1/p^e]}$ on $X$ is defined by $\mathrm{Tr}_F^e(F^e_*(\bb \cdot \omega_X))=\bb^{[1/p^e]} \cdot \omega_X$.  
Then
\[
\tau(X, \ba^t)=\bigcup_{e \in \N} (\ba^{\lceil t p^e \rceil})^{[1/p^e]}=(\ba^{\lceil t q \rceil})^{[1/q]}
\]
for a sufficiently large power $q$ of $p$. 
\end{prop}

Let $X$ be a normal projective variety over a perfect field of characteristic $p>0$ and $\Delta$ be an effective $\Q$-divisor on $X$ such that $K_X+\Delta$ is $\Q$-Cartier. 
Let $D$ be a $\Q$-Cartier $\Q$-divisor on $X$ such that $rD$ is Cartier and $H^0(X, \sO_X(rD)) \ne 0$ for some $r \in \N$. 
We fix such an $r \in \N$, and let $\bb_m \subseteq \sO_X$ denote the base ideal of the linear series $|mrD|$, that is, $\bb_m$ is the image of the natural map 
\[H^0(X, \sO_X(mrD)) \otimes \sO_X(-mrD) \to\sO_X\]
for every $m \in \N$. 
Then $\bb_{\bullet}=\{\bb_m\}_{m \in \N}$ is a graded family of ideal sheaves on $X$. 
For every real number $\lambda >0$, the asymptotic test ideal $\tau((X, \Delta), \bb_{\bullet}^{\lambda})$ associated to the triple $((X, \Delta), \bb_{\bullet}^{\lambda})$ is defined in a similar way to the asymptotic test ideal $\tau(\ba_{\bullet}^t)$ in \S\ref{symbolic power}: $\tau((X, \Delta), \bb_{\bullet}^{\lambda})$ is the unique maximal element of the family of test ideals $\{\tau((X, \Delta), \bb_m^{\lambda /m})\}_{m \in \N}$. 
We put $\tau((X, \Delta), \lambda \cdot ||D||)=\tau((X, \Delta), \bb_{\bullet}^{\lambda/r})$, and denote $\tau((X, \Delta), \lambda \cdot ||D||)$ simply by $\tau(\lambda \cdot ||D||)$ when $\Delta=0$. 
The ideal sheaf $\tau((X, \Delta), \lambda \cdot ||D||)$ is independent of the choice of $r$. 

Making use of the asymptotic test ideal $\tau((X, \Delta),\lambda \cdot ||D||)$ instead of the asymptotic multiplier ideal $\mathcal{J}((X, \Delta), \lambda \cdot ||D||)$, we can show a positive characteristic analogue of \cite[Corollary 11.2.13]{La}.

\begin{thm}[cf.~\cite{Sch6}]\label{base point free}
Suppose that $\sO_X(H)$ is a globally generated ample line bundle on $X$. 
Let $\lambda>0$ be a rational number and $L$ be a Cartier divisor on $X$ such that $L-(K_X+\Delta)-\lambda D$ is nef and big. 
Then for every integer $n \ge d=\dim X$, $\tau((X, \Delta), \lambda \cdot ||D||) \otimes \sO_X(L+nH)$ is globally generated.
\end{thm}

We give a sketch of the proof. 
For simplicity, we assume that $X$ is a nonsingular projective variety, $\Delta=0$, $D$ is a Cartier divisor such that $H^0(X, \sO_X(D)) \ne 0$ and $L-K_X-\lambda D$ is ample. 
Let $\bb_{m}$ be the base ideal of the linear series $|mD|$. 
It then follows from the definition of the test ideal $\tau(\lambda \cdot ||D||)$ and Proposition \ref{smooth test ideal} that 
\[
\tau(\lambda \cdot ||D||)=\tau(X, \bb^{\lambda/m}_{m})=(\bb_m^{\lceil \lambda q/m \rceil})^{[1/q]}
\]
for sufficiently divisible $m$ and sufficiently large $q=p^e$. 
By the definition of $(\bb_m^{\lceil \lambda q/m \rceil})^{[1/q]}$, the trace map $\mathrm{Tr}_F^e$ induces a surjective map 
\[F^e_*(\bb_m^{\lceil \lambda q/m \rceil} \otimes \omega_X) \to \tau(\lambda \cdot ||D||) \otimes \omega_X.\]
Tensoring this with $\sO_X(L-K_X+nH)$, one has a surjection  
\[F^e_*\left(\bb_m^{\lceil \lambda q/m \rceil} \otimes \sO_X(qL-(q-1)K_X+qnH)\right) \to \tau(\lambda \cdot ||D||) \otimes \sO_X(L+nH).\]

On the other hand, the natural map $H^0(X, \sO_X(mD)) \otimes \sO_X(-mD) \to\bb_m$ is surjective by the definition of $\bb_m$. 
Therefore, its $\lceil \lambda q/m \rceil$-th symmetric product 
\[\mathrm{Sym}^{\lceil \lambda q/m \rceil} H^0(X, \sO_X(mD)) \otimes \sO_X(-m \lceil \lambda q/{m}\rceil D)\to \bb^{\lceil \lambda q/m \rceil}_m\]
is also a surjection. 
Put $W:=\mathrm{Sym}^{\lceil \lambda q/m \rceil} H^0(X, \sO_X(mD))$. 
The above two surjections induce a surjection 
\[
W \otimes F^e_*\sO_X(qL-(q-1)K_X-m \lceil \lambda q/{m} \rceil D+qnH) \to  \tau(\lambda \cdot ||D||) \otimes \sO_X(L+nH).\]
Thus, it suffices to show that $F^e_*\sO_X(qL-(q-1)K_X-m \lceil \lambda q/{m} \rceil D+qnH)$ is globally generated. 
Indeed, it follows from Lemma \ref{Mumford's lemma}, because 
\[H^i(X, \sO_X(qL-(q-1)K_X-m \lceil \lambda q/{m} \rceil D+q(n-i)H))=0\]
for all $i>0$ by Fujita's vanishing theorem (Lemma \ref{Fujita vanishing}). 

\begin{lem}[\textup{\cite[Theorem 1.8.5]{La1}}]\label{Mumford's lemma}
Let $X$ be a projective variety and $\mathcal{F}$ be a coherent sheaf on $X$. 
Suppose that $\sO_X(H)$ is a globally generated ample line bundle on $X$. 
If $H^i(X, \mathcal{F} \otimes \sO_X(-iH))=0$ for all $i>0$, then $\mathcal{F}$ is globally generated. 
\end{lem}

\begin{lem}[\cite{Fu}, cf.~\textup{\cite[Theorem 1.4.35]{La1}}]\label{Fujita vanishing}
Let $X$ be a projective scheme over a field $k$ and $H$ be an ample Cartier divisor on 
$X$. 
Given any coherent sheaf $\mathcal{F}$ on $X$, there exists an integer $m(F, H)$ such that
$H^i(X, \mathcal{F} \otimes \sO_X(mH+D))=0$ for all $i>0$, $m \ge m(F, H)$, and any nef Cartier divisor $D$ on $X$.
\end{lem}

\begin{rem}[\cite{Ha2}]
Using a very similar argument to the proof of Theorem \ref{base point free}, we can show a special case of Fujita's conjecture in positive characteristic: 
Let $X$ be a $d$-dimensional projective variety over a perfect field of characteristic $p>0$ such that every local ring $\sO_{X, x}$ of $X$ is $F$-injective. 
If $\mathcal{L}$ is a globally generated ample line bundle on $X$, then $\omega_X \otimes \mathcal{L}^{\otimes d+1}$ is also globally generated. 
\end{rem}

We conclude this section by mentioning an application of Theorem \ref{base point free}. 
In order to state the result, we need to introduce some notation. 

Let $X$ be a nonsingular projective variety over an algebraically closed field of characteristic $p>0$ and $D$ be a $\Q$-divisor on $X$. 
Fix an $r \in \N$ such that $rD$ is an integral divisor. 
The \textit{stable base locus} $\mathbf{B}(D)$ of $D$ is $\bigcap_{n \in \N} \mathrm{Bs}(nrD)$, where $\mathrm{Bs}(nrD)$ denotes the base locus of the linear series $|nrD|$ (with reduced scheme structure).  
Since $\mathbf{B}(D)=\mathrm{Bs}(n rD)$ for sufficiently divisible $n$ by Noetherian property, $\mathbf{B}(D)$ is independent of the choice of $r$. 
The \textit{asymptotic base locus} $\mathbf{B}_{-}(D)$\footnote{This locus is often referred to as the \textit{non-nef locus} or the \textit{restricted base locus}.} is then defined by $\mathbf{B}_{-}(D)=\bigcup_A \mathbf{B}(D+A)$, where $A$ runs through all ample $\Q$-divisors on $X$. 
It follows from the definition that $\mathbf{B}_{-}(D)$ depends only on the numerical equivalence class of $D$ and that $\mathbf{B}_{-}(D)=\emptyset$ if and only if $D$ is nef (see \cite{ELMNP}). 

Next, we introduce the notion of asymptotic orders of vanishing. Let $x \in X$ be a closed point and $\m_{x}$ denote the maximal ideal of the local ring $\sO_{X, x}$. 
Suppose that $H^0(X, \sO_X(rD)) \ne 0$ for the above $r \in \N$. 
For each $m \in \N$, we denote by $\mathfrak{b}_m \subseteq \sO_X$ the base ideal of the linear series $|mrD|$ and by $\mathrm{ord}_{x}(\mathfrak{b}_m)$ the order of vanishing of $\mathfrak{b}_m$ at $x$. 
In other words, $\mathrm{ord}_{x}(\mathfrak{b}_m)=\max\{\nu \ge 0 \; | \; \mathfrak{b}_{m,x} \subseteq \m_{x}^{\nu}\}$.
The asymptotic order of vanishing $\mathrm{ord}_x(||D||)$ is then defined by
\[\mathrm{ord}_x(||D||)=\inf_{m \in \N}\frac{\mathrm{ord}_x(\mathfrak{b}_m)}{mr}=\lim_{m \to \infty}\frac{\mathrm{ord}_x(\mathfrak{b}_m)}{mr},\]
where the last equality follows from the fact that $\{\bb_m\}_{m \in \N}$ is a graded family of ideal sheaves on $X$. 
The definition of $\mathrm{ord}_x(||D||)$ is independent of the choice of $r$. 

Replacing the role of asymptotic multiplier ideals in \cite{ELMNP} by asymptotic test ideals and  \cite[Corollary 11.2.13]{La} by Theorem \ref{base point free}, we obtain the following theorem. 

\begin{thm}[\cite{Mu2}]
Suppose that $D$ is a big $\Q$-divisor on $X$ and $x \in X$ is a closed point. 
Then $x \notin \mathbf{B}_{-}(D)$ if and only if $\mathrm{ord}_x(||D||)=0$. 
\end{thm}


\section{Hilbert-{K}unz multiplicity}\label{HK section}

The Hilbert-Samuel multiplicity is a fundamental invariant in commutative ring theory and singularity theory. 
There is another kind of multiplicity in positive characteristic, Hilbert-{K}unz multiplicity, introduced by Paul Monsky \cite{Mo1}.\footnote{Hilbert-Kunz multiplicity has origin in the work of Kunz \cite{Ku}, \cite{Ku2}, but he erroneously thought that it did not exist in general. 
Monsky proved in \cite{Mo1} that it always exists and named it after Kunz.} 
Hilbert-{K}unz multiplicity has properties both similar to and different from those of Hilbert-Samuel multiplicity. 
For example, Hilbert-Samuel multiplicity determines membership in integral closure, and Hilbert-{K}unz multiplicity determines membership in tight closure. 
Hilbert-Samuel multiplicity is a positive integer, but Hilbert-{K}unz multiplicity is not necessarily even a rational number (see \cite{Br2}). 

When the second-named author began the study of Hilbert-Kunz multiplicity,  Monsky had found some mysterious behavior of it (see \cite{HaMo} for example) and it had been thought of as a difficult and elusive subject to study. 
Watanabe had a chance to meet Monsky in 1997 and said,  ``I want to make Hilbert-{K}unz multiplicity less mysterious." 
Monsky then answered,  ``I want to make Hilbert-{K}unz multiplicity \textit{more} mysterious!" 
Indeed, the counterexample to the localization problem in tight closure theory, mentioned in Remark \ref{localization remark}, has come from Monsky's study of Hilbert-{K}unz multiplicity. 

In this section, we overview the theory of Hilbert-Kunz multiplicity. We recommend \cite{Hu2} \S 6 for a nice introduction to Hilbert-Kunz multiplicity. 
The authors of this article got a lot of inspiration from this book. 

Suppose that  $(A, \fkm)$ is a local ring of characteristic $p>0$ and $I$ is an $\fkm$-primary ideal of $A$. 
The Hilbert-{K}unz multiplicity of $I$ is defined as follows.  
We also recall the definition of Hilbert-Samuel multiplicity of $I$. 

\begin{defn}[\cite{Mo1}]\label{HKdef}
Let $I$ be an $\fkm$-primary ideal of a $d$-dimensional local ring $(A, \m)$ of characteristic $p>0$.
\begin{enumerate}
\item 
The \textit{Hilbert-Samuel multiplicity} $e(I)$ of $I$ is defined by 
 \[
e(I):=\lim_{n \to \infty} d! \frac{\ell_A (A/I^n) }{n^d}.
\]
\item 
The \textit{Hilbert-{K}unz multiplicity} $e_{\rm HK}(I)$ of $I$ is defined by 
\[
e_{\rm HK}(I):=\lim_{e \to \infty} \frac{\ell_A (A/I^{[p^e]}) }{p^{ed}}.
\]
\end{enumerate}
The limits $e(I)$ and $e_{\rm HK}(I)$ always exist. 
If $I=\fkm$, we denote $e(A):=e(\fkm)$ (respectively, $e_{\rm HK}(A):=e_{\rm HK}(\fkm)$)
 and call it the Hilbert-Samuel multiplicity (respectively, Hilbert-{K}unz multiplicity) of $A$. 
\end{defn}

\begin{rem} 
Brenner \cite{Br2} has recently given an example of a local domain whose Hilbert-Kunz multiplicity is irrational. 
Monsky conjectured in \cite{Mo3} that if $A=\F_2[[U,V,X,Y,Z]]/(UV + X^3+Y^3+XYZ)$, then $e_{\rm HK}(A) =4/3 + 5/(14\sqrt{7})$. This conjecture is still open. 
On the other hand, if $R=\bigoplus_{n \ge 0}R_n$ is a two-dimensional graded domain with $R_0$ an   algebraically closed field of characteristic $p>0$, then $e_{\rm HK}(I)$ is a rational number for all homogeneous $R_+$-primary ideals $I$ (see \cite{Br}, \cite{Tr}). 
If $(A, \m)$ is a local ring of finite $F$-representation type (see \S \ref{closing} for rings of finite $F$-representation type), then $e_{\rm HK}(I)$ is a rational number for all $\m$-primary ideals $I$ (see \cite{Se}, \cite{Wa4}).
\end{rem}

$e(I)$ and $e_{\rm HK}(I)$ do not change after passing to completion by definition. 
Also, removing the lower-dimensional irreducible components of $\Spec R$ does not affect the values of $e(I)$ and $e_{\rm HK}(I)$: if $\fka$ is the intersection of all prime ideals $\fkp$ of $A$ with $\dim A/\fkp=d$, then $e(I)=e(I (A/\fka))$ and $e_{\rm HK}(I)=e_{\rm HK}(I (A/\fka))$. 
Thus, we may assume that $A$ is complete and equidimensional. 

\begin{assumption}
We assume from now on that $(A, \m)$ is a $d$-dimensional complete equidimensional reduced local ring of characteristic $p>0$, unless stated otherwise. 
\end{assumption} 

Given $\m$-primary ideals $I \subseteq I'$, it is obvious from the definition that $e(I') \le e(I)$ and 
$e_{\rm HK}(I')\le e_{\rm HK}(I)$.  
Then it is natural to ask when the equalities hold. 
These equality conditions are described in terms of integral closure and tight closure. 

\begin{prop}\label{HK=} 
For $\fkm$-primary ideals  $I \subset I'$ of $A$, the following holds.
\begin{enumerate}
\item  $($\cite{Re}$)$ $e(I') = e(I)$ if and only if $I' \subset \overline{I}$. 
\item  $($\cite{HH0}$)$  $e_{\rm HK}(I')= e_{\rm HK}(I)$ if and only if $I'\subset I^*$.
\end{enumerate}
\end{prop}
 
Next, we observe the relationship of $e(I)$ and $e_{\rm HK}(I)$. 
\begin{lem}\label{HKineq} 
Let $I$ be an $\fkm$-primary ideal $I $ of $A$. 
\begin{enumerate}
\item 
We have the following inequalities: 
\[\frac{1}{d !}e(I) \leq e_{\rm HK}(I) \leq e(I).\]
In particular, if $d=1$, then $e(I)= e_{\rm HK}(I)$. 
When $d \ge 2$, the left inequality is strict $($\cite{Han}$)$, but there are examples where $(1/{d !})e(I)$ is arbitrary close to $e_{\rm HK}(I)$ $($Example $\ref{WY example}$ $(2))$.

\item $($\cite{Le}$)$ If $I$ is a parameter ideal, then $e(I)= e_{\rm HK}(I)$. 
\item If $A$ is regular, then $e_{\rm HK}(I)=\ell_A(A/I)$ by the flatness of the Frobenius map $($Theorem $\ref{Kunz})$. In particular, if $A$ is regular, then $e_{\rm HK}(A)=1$. 
\end{enumerate}
\end{lem}
 
 It follows from Proposition \ref{HK=} and Lemma \ref{HKineq} that if $A$ is $F$-rational but not regular, then $e_{\rm HK}(A) \lneq e(A)$. 
For the proof of this,  we may assume that the residue field $A/\m$ is infinite.  
 Then it is well-known that there exists an $\m$-primary parameter ideal $J \subset A$ such that $\overline{J}=\m$. 
 Since $J^*=J \subsetneq \m$ by assumption, one has  $e_{\rm HK}(A) \lneq e_{\rm HK}(J)=e(J)=e(A)$.

Since $e(I)= e_{\rm HK}(I)$ for one-dimensional local rings as we have seen in Lemma \ref{HKineq} (1), we mainly consider the case where $d \ge 2$. 

\begin{ex} [\cite{WY1}, \cite{WY2}]\label{WY example}
Let $k$ be the residue field of $A$ and $I$ be an $\m$-primary ideal of $A$. 

(1) Let $A\subset  B$ be a finite extension of local domains with the same residue field. 
Then $e_{\rm HK}(I) ={e_{\rm HK}(I B)}/{r}$, where $r$ is the rank of $B$ as an $A$-module.  
In particular, if  $B$ is regular, then $e_{\rm HK}(I) = \ell_B (B/IB)/{r}$. 

(2) Let $A$ be the $r$-th Veronese subring of $B=k[[X_1,\ldots ,X_d]]$, that is, the subring of $B$ generated by all monomials of degree $r$. 
Then by (1), we have $e_{\rm HK}(A) =\ell_B( B /\fkm B)/{r} =\binom{d+r-1}{d}/r$. 
On the other hand, since $e(A) = r^{d-1}$,  
one has $\lim_{r\to \infty} e_{\rm HK}(A)/ e(A)  ={1}/{d!}$. 
Thus, this is an example where $({1}/{d !})e(I)$ is arbitrary close to $e_{\rm HK}(I)$. 

(3) If $A = k[[X,Y,Z]]/(XY - Z^n)$ where $n \ge 2$ is an integer, then $e_{\rm HK}(A) = 2 - {1}/{n}$. 
In general, when $A$ is a two-dimensional $F$-rational Gorenstein complete local ring with algebraically closed coefficient field $k$, there exists a linearly reductive finite subgroup $G$ of ${SL}(2, k)$ such that $A$ is isomorphic to the completion of the ring of invariants $k[X, Y]^G$ (see \cite{Has1}).  
Then $e_{\rm HK}(A) = 2 - {1}/{|G|}$. 

(4)  Assume that $A$ is a two-dimensional unmixed local ring.\footnote{A $d$-dimensional local ring $(R, \m)$ is called \textit{unmixed} if for every associated prime $\fkp$ of the $\m$-adic completion $\widehat{R}$, one has $\dim\widehat{R}/\fkp =d$.}  
Then for every $\fkm$-primary 
ideal $I$, we have 
$e_{\rm HK}(I)  \ge (e(I) +1) /{2}$.  If  $k$ is an algebraically closed field, 
then the equality holds when $I=\m$ if and only if the associated graded ring $\mathrm{gr}_{\fkm} \,A =
\bigoplus_{n\ge 0} \fkm^n/\fkm^{n+1}$ is isomorphic to the $e(A)$-th Veronese subring of 
 $k[X,Y]$.
 In particular, if $A$ is not regular, then 
$e_{\rm HK}(A)  \ge {3}/{2}$ and the equality holds if and only if 
$A\cong k[[X,Y,Z]]/(XY-Z^2)$. 
\end{ex}
 
\begin{rem}\label{DCC?} Fix an integer $d \ge 2$. 
We see from Example \ref{WY example} (3) that 
there is a sequence of two-dimensional local rings $\{A_n\}$ such that 
$e_{\rm HK}(A_n)$ converges to 2 from below. 
Is there a a sequence of local rings $\{A_n\}$ of dimension $d$ such that 
$e_{\rm HK}(A_n)$ converges to some real number from above?  
\end{rem}

The following theorem of Nagata is very fundamental in commutative ring theory and singularity theory. 
\begin{thm}[\textup{\cite[(40.6)]{Na}}]\label{Nagata}
If $A$ is an unmixed local ring, then $A$ is regular if and only if $e(A)=1$. 
\end{thm}

The analogous statement for Hilbert-{K}unz multiplicity is also true. 
\begin{thm}[\cite{WY1}, \cite{HuYa}]\label{HK-reg} 
If $A$ is an unmixed local ring of positive characteristic, 
then $A$ is regular if and only if $e_{\rm HK}(A)=1$. 
\end{thm}
 
We give a sketch of the proof of Theorem \ref{HK-reg}. 
The key ingredient is the following fact: if $e_{\rm HK}(A)=1$, then $e_{\rm HK}(J)-\ell_A(I/J) \le  e_{\rm HK}(I) \le \ell_A(A/I^*)$ for any $\fkm$-primary ideals  $J \subset I$. 
It follows from this fact and Theorem \ref{GN thm} that $\ell_A(A/\mathfrak{q}^*)=e_{\rm HK}(\mathfrak{q})=e(\mathfrak{q})$ for every $\m$-primary parameter ideal $\mathfrak{q} \subset A$, which implies that $A$ is $F$-rational. 
Put  $I=\m^{[p]}$ and take an $\m$-primary parameter ideal $J \subset I$. 
By the above fact and the Cohen-Macaulay property of $A$, we see that 
\begin{align*}
\ell_A(A/I)=\ell_A(A/J)-\ell_A(I/J)=e_{\rm HK}(J)-\ell_A(I/J) \le e_{\rm HK}(I) &\le \ell_A(A/I^*)\\ 
&\le \ell_A(A/I).
\end{align*}
Namely, $\ell_A(A/I)=e_{\rm HK}(I)=p^d e_{\rm HK}(A)=p^d$, and we conclude from Theorem 
\ref{Kunz} (3)$\Rightarrow$(1) that $A$ is regular. 
 
Theorem \ref{GN thm} was proved under some additional assumption in \cite{WY1}, and in the general case in \cite{GN}.  
It should be compared with the basic fact that for an $\m$-primary parameter ideal $J \subset A$, one has $\ell_A (A/J) \ge e(J)$ and the equality holds if $A$ is Cohen-Macaulay. 
Replacing $J$ with $J^*$, we have the reverse inequality. 

\begin{thm}[\cite{GN}, cf.~\cite{WY1}]\label{GN thm} 
If $J$ be an $\m$-primary parameter ideal of $A$, then $\ell_A (A/J^*) \le e(J)$. 
Moreover, if $A$ is unmixed and $\ell_A (A/J^*) = e(J)$ for some $\m$-primary parameter ideal $J \subset A$, then $A$ is $F$-rational.
\end{thm}
 
By Theorem \ref{HK-reg}, if an unmixed local ring  $A$ is not regular, then $e_{\rm HK}(A) >1$. 
It is natural to ask whether there is a sharp lower bound for Hilbert-Kunz multiplicity of $d$-dimensional non-regular unmixed local rings.

\begin{conj}[\cite{WY3}]\label{HKConj} Let $k$ be a field of characteristic $p \ge 3$ and $d \ge 2$ be an integer.  If we put 
\[
A_{p, d} = k[[X_0, X_1, \ldots , X_d]]/ (X_0^2+X_1^2+ \ldots + X_d^2), 
\]
then the following holds for an arbitrary $d$-dimensional unmixed local ring $(A, \m)$ of characteristic $p$ with residue field $k$. 
\begin{enumerate}
\item If $A$ is not regular, then $e_{\rm HK}(A) \ge e_{\rm HK}(A_{p,d})$. 
\item If $e_{\rm HK}(A) = e_{\rm HK}(A_{p,d})$ and if $k$ is an algebraically closed field, 
then the $\m$-adic completion $\widehat{A}$ of $A$ is isomorphic to $A_{p,d}$. 
\end{enumerate}
\end{conj}
 
In Conjecture \ref{HKConj}, (2) is known for $d\le 4$ (\cite{WY3}), and (1) is known for $d \le 6$ (\cite{AbE}) or when $A$ is a complete intersection  (\cite{EnSh}).

It is known that $e_{\rm HK}(A_{p,2})={3}/{2}$ and $e_{\rm HK}(A_{p,3})={4}/{3}$ for all $p$. 
If $d\ge 4$, then the value of $e_{\rm HK}(A_{p,d})$ depends on $p$. 
Indeed, $e_{\rm HK}(A_{p,4})={(29p^2+15)}/{(24p^2+12)}$ and $\lim_{p \to \infty} e_{\rm HK}(A_{p,4})={29}/{24}$.
One of the most mysterious properties of $A_{p,d}$ is that the limit $\lim_{p \to \infty}e_{\rm HK}(A_{p,d})$ can be described by using the Maclaurin expansion of $\tan x + \sec x$. 
 
\begin{thm}[\cite{GM}]\label{SecTan}  
Let the notation be the same as that in Conjecture $\ref{HKConj}$. 
Let $\sum_{d\ge 0} ({c_d}/{d!})  x^d$ be the Maclaurin expansion of $\tan x + \sec x$. 
Then 
\[
\lim_{p\to \infty} e_{\rm HK}(A_{p, d}) =
1 + \frac{c_d}{d!}.
\]
\end{thm}
 
We conclude this section by mentioning Hilbert-Kunz function. 
The \textit{Hilbert-Kunz function} of an $\m$-primary ideal $I$ of $A$ is a function $\Z_{\ge 0} \to \R$ sending $e$ to $\ell_A(A/I^{[p^e]})$. 
When Monsky defined Hilbert-Kunz multiplicity in \cite{Mo1}, he showed that $\ell_A(A/I^{[q]})=e_{\rm HK}(I)q^d+O(q^{d-1})$ for $q=p^e$. 
If $A$ is normal, then we can strengthen this result: Let $(A, \m)$ be a $d$-dimensional excellent normal local ring of characteristic $p>0$ with perfect residue field.  
Then Huneke-McDermott-Monsky \cite{HuMcMo} proved that for every $\fkm$-primary ideal $I \subset A$, there exists a real number $\beta(I)$ such that
 \[
 \ell_A (A/I^{[q]}) = e_{\rm HK}(I) q^d + \beta(I) q^{d-1} + O(q^{d-2}).\footnote{
One might hope to generalize this result to show that there exists a real number $\gamma(I)$ such that  $\ell_A (A/I^{[q]}) = e_{\rm HK}(I) q^d + \beta(I) q^{d-1} + \gamma(I) q^{d-2} + O(q^{d-3})$. 
However, this cannot be true because of Han-Monsky's example \cite{HaMo}.}
\]
In general $\beta(I) \ne 0$, but if $A$ is a complete $\Q$-Gorenstein normal local ring, then  $\beta(I) = 0$ (see \cite{Kura}). 
  
 A geometric description of Hilbert-Kunz function (or multiplicity) is desired for further study. 
 When $R=\bigoplus_{n \ge 0}R_n$ is a two-dimensional standard graded normal domain with $R_0$ an algebraically closed field of characteristic $p>0$ and $I$ is a homogeneous $R_+$-primary ideal, 
 an explicit description of the Hilbert-{K}unz multiplicity was given in \cite{Br}, \cite{Tr} by using vector bundles on $X=\Proj R$. 
 
 
\section{Closing remark}\label{closing}
There are several important topics on $F$-singularities that we have not been able to include because of space limitations. 
 
 Let $(R, \m)$ be a $d$-dimensional complete reduced local ring with prime characteristic $p$ and perfect residue field $R/\m$. 
 Then the \textit{$F$-signature} $s(R)$ of $R$ is defined by $s(R)=\lim_{e \to \infty}a_e/p^{ed}$, where $a_e$ denotes the largest rank of a free $R$-module appearing in a direct sum decomposition of $F^e_*R$.
This invariant was first introduced by Huneke-Leuschke \cite{HuLe} and its existence was proved by Kevin Tucker \cite{Tuc} in full generality. 
The $F$-signature can be characterized in terms of the Hilbert-{K}unz multiplicity: $s(R)$ coincides with the infimum of $e_{\rm HK}(I) - e_{\rm HK}(I')$ as $I\subsetneq I'$ run through all $\m$-primary ideals of $R$.\footnote{The Gorenstein case was proved in \cite{WY4} and the general case was recently announced by Kevin Tucker.} 
Also, the $F$-signature $s(R)$ measures the singularities of $R$. It takes a value between 0 and 1, and $s(R)=1$ (respectively, $s(R)>0$) if and only if $R$ is regular (respectively, strongly $F$-regular) (see \cite{AbLe}). 

The notion of rings with \textit{finite $F$-representation type} was introduced by Smith-Van den Bergh \cite{SvdB}, inspired by the notion of rings with finite Cohen-Macaulay type. 
They fit into the theory of $D$-modules in positive characteristic and satisfy some finiteness properties (\cite{Ya}, \cite{TaTa}). 
For example, the $F$-signature of a ring with finite $F$-representation type is a rational number. 
Affine toric rings and rings of invariants under linearly reductive group actions are examples of rings with finite $F$-representation type. 
The relationship with other $F$-singularities treated in this article is not clear at the moment. 
There is an example of a ring that is strongly $F$-regular but not with finite $F$-representation type (\cite{SiSw}). 
 
\textit{Global $F$-regularity} was introduced by Smith \cite{Sm3} and can be viewed as a global version of strong $F$-regularity.  
It is a global property of a projective variety over a field of positive characteristic. 
For example, the anti-canonical divisor of a global $F$-regular variety is big (\cite{SS}). 
Projective toric varieties and Schubert varieties are examples of globally $F$-regular varieties (\cite{Lau}, \cite{Has0}). 
 As a global version of the correspondence between strong $F$-regularity and being klt, it is conjectured in \cite{SS} that a normal projective variety $X$ over an algebraically closed field of characteristic zero is log Fano if and only if it is of ``globally $F$-regular type."
 This conjecture is known to be true if $\dim X=2$ (\cite{Ok}) or if $X$ is a Mori dream space (\cite{GOST}).

\end{document}